\documentclass[12pt]{amsart}
\usepackage{amsmath,amsthm,amsfonts,latexsym,amssymb}
\usepackage{graphicx}

\calclayout\evensidemargin .01in

\newcommand{\C}{\mathbb{C}}
\newcommand{\R}{\mathbb{R}}

\newtheorem{theorem}{Theorem}[section]
\newtheorem{lemma}[theorem]{Lemma}
\newtheorem{proposition}[theorem]{Proposition}
\newtheorem{corollary}[theorem]{Corollary}

\theoremstyle{definition}

\newtheorem{remark}[theorem]{Remark}
\newtheorem{notation}[theorem]{Notation}
\newtheorem{conventionandnotation}[theorem]{Convention \& Notation}
\newtheorem{fact}[theorem]{Fact}
\newtheorem{definition}[theorem]{Definition}

\title[Stabilizations via Lefschetz Fibrations and Exact Open Books]
{Stabilizations via Lefschetz Fibrations and \\Exact Open Books}

\author{Selman Akbulut}
\address{Department of Mathematics, Michigan State University, Lansing MI, USA}
\email{akbulut@math.msu.edu}
\thanks{The first author is partially supported by NSF FRG grant DMS- 0905917}

\author{M. Firat Arikan}
\address{Max Planck Institute for Mathematics, Bonn, GERMANY}
\email{arikan@mpim-bonn.mpg.de}
\thanks{The second author is partially supported by NSF FRG grant DMS-1065910}

\subjclass[2000]{57R65, 58A05, 58D27}
\keywords{Contact \& symplectic structures, open book, Lefschetz fibration, stabilization}
\date{\today}

\begin{document}

\maketitle

\begin{abstract}
We show that if a contact open book $(\Sigma,h)$ on a $(2n+1)$-manifold $M$ ($n\geq1$) is induced by a Lefschetz fibration $\pi:W \to D^2$, then there is a one-to-one correspondence between positive stabilizations of $(\Sigma,h)$ and \emph{positive stabilizations} of $\pi$. More precisely, any positive stabilization of $(\Sigma,h)$ is induced by the corresponding positive stabilization of $\pi$, and conversely any positive stabilization of $\pi$ induces the corresponding positive stabilization of $(\Sigma,h)$. We define \emph{exact open books} as boundary open books of compatible exact Lefschetz fibrations, and show that any exact open book carries a contact structure. Moreover, we prove that there is a one-to-one correspondence (similar to the one above) between \emph{convex stabilizations} of an exact open book and \emph{convex stabilizations} of the corresponding compatible exact Lefschetz fibration. We also show that convex stabilization of compatible exact Lefschetz fibrations produces symplectomorphic completions.
\end{abstract}


\section{Introduction}

In the last decade the correspondence given by Giroux \cite{Giroux}, between contact structures and open book decompositions have led to many developments in understanding the relations between the contact geometry and the topology of the underlying odd dimensional closed manifolds. This correspondence is much stronger in dimension three and has been used as a bridge between four dimensional geometries and topology, leading much progress in understanding of different types of fillability and Lefschetz type fibrations.

\vspace{.05in}

One of the main features used in the above correspondence is positive stabilization. Namely, if we positively stabilize an open book $(\Sigma,h)$ carrying a contact structure $\xi$ on a closed $3$-manifold $M$, then the resulting open book still carries $\xi$. Such stabilizations can be interpreted as taking the contact connect sum of $(M,\xi)$ with $(S^3,\xi_{st})$ where $\xi_{st}$ is the unique tight (Stein fillable) contact structure on the $3$-sphere $S^3$. In terms of open books, this corresponds to taking the Murasugi sum (or plumbing) of $(\Sigma,h)$ with the open book $(H^+, \tau_C)$ on $S^3$ where $H^+$ is the positive (left-handed) Hopf band and $\tau_C$ denotes the right-handed Dehn twist along the core circle $C$ in $H^+$.

\vspace{.05in}

To get analogous statements for higher dimensions, one can replace $(H^+, \tau_C)$ with its generalization $\mathcal{OB}$, which is an open book carrying the standard contact structure $\xi_0$ on $(2n+1)$-sphere $S^{2n+1}$ and obtained from a certain Milnor fibration. The pages of $\mathcal{OB}$ are diffeomorphic to the closed tangent unit disk bundle $\mathcal{D}(TS^n)$ over $S^n$ and its monodromy is the (generalized) right-handed Dehn twist along the zero section (see below). Then one can define a positive stabilization of an open book $(\Sigma,h)$ carrying a contact structure $\xi$ on an $(2n+1)$-dimensional closed manifold $M^{2n+1}$ by taking the Murasugi sum of $(\Sigma,h)$ with $\mathcal{OB}$, along a properly embedded Lagrangian $n$-ball $L$ in $\Sigma $ with Legendrian boundary and a fiber in $\mathcal{D}(TS^n)$. Again this amounts to taking the contact connect sum of $(M^{2n+1},\xi)$ with $(S^{2n+1},\xi_0)$ and stabilized open book still carries $\xi$ \cite{Giroux}. In terms of contact surgery and Weinstein handles, a positive stabilization corresponds to performing (resp. attaching) a pair of subcritical and critical surgeries (resp. Weinstein handles) which cancels each other (see \cite{Van Koert} for a proof).

\vspace{.05in}

One of the missing part of this picture is the relation of such operations to Lefschetz fibrations. The aim of the present work is to provide some results to fill this gap. Through out the paper, the base space of any Lefschetz fibration is assumed to be the $2$-disk, and we focus only on the open books which are induced by Lefschetz fibrations. We study the open book $\mathcal{OB}$ (which is induced by a certain Lefschetz fibration $\mathcal{LF}$ on the standard $(2n+2)$-ball) in Section \ref{sec:Preliminaries} where we also recall positive stabilizations of open books and the characterization of Lefschetz fibrations. In Section \ref{sec:Positive stabilization of Lefschetz fibrations}, we explicitly define a process, called \emph{positive stabilization} on Lefschetz fibrations and show that there is a one-to-one correspondence between positive stabilizations of open books and Lefschetz fibrations. \emph{Exact open books} are introduced in Section \ref{sec:Exact_Open_Books} as boundaries of compatible exact Lefschetz fibrations.  After recalling Weinstein handles and isotropic setups briefly in Section \ref{sec:Isotropic Setups and Weinstein Handles}, we will get a similar correspondence for exact open books and compatible exact Lefschetz fibrations in Section \ref{sec:Convex Stabilization}, where we also define \emph{convex stabilization} as an exact symplectic version of positive stabilization. We remark that any observation we will make here is also true for dimensions $3$ and $4$, and so the work done here can be thought as canonical generalizations of the corresponding $3$- and $4$-dimensional results to higher dimensions.


\section{Preliminaries} \label{sec:Preliminaries}

\subsection{The open book $\mathcal{OB}$ and the associated Lefschetz fibration $\mathcal{LF}$}
Consider the polynomial $P$ on the complex space $\mathbb{C}^{n+1}$ (for $n\geq 1$) given by $$P(z_1,...,z_{n+1})=z_1^2+z_2^2+ \cdots +z_{n+1}^2.$$
It is clear that the only critical point of $P$ occurs at the origin, and so the intersection of the zero set $Z(P)$ of $P$ with the sphere $$\mathbb{S}_{\varepsilon}^{2n+1}=\{|z_1|^2+|z_2|^2+ \cdots + |z_{n+1}|^2=\varepsilon^2\},$$
where $\varepsilon>0$ is small enough, is a smooth manifold $K$ of dimension $2n-1$. $K$ is a member of a family known as Brieskorn manifolds introduced in \cite{Brieskorn}. It is known by \cite{Milnor} that the complement $\mathbb{S}_{\varepsilon}^{2n+1} \setminus K$ of $K$ fibers over the unit circle $S^1 \subset \mathbb{C}$ via the map $\Theta:\mathbb{S}_{\varepsilon}^{2n+1} \setminus K \to S^1$ given by $$\Theta(z_1,...,z_{n+1})=\frac{P(z_1,...,z_{n+1})}{|P(z_1,...,z_{n+1})|}.$$

Let $\mathcal{OB}$ be the open book on $\mathbb{S}_{\varepsilon}^{2n+1}$ determined by the pair $(K,\Theta)$. For any $e^{i\theta} \in S^1$, the Milnor fiber (or the {\em page} of $\mathcal{OB}$) $F_{\theta}:=\Theta^{-1}(e^{i\theta})$ is parallelizable and has the homotopy type of $S^n$ \cite{Milnor}, and indeed it can be identified with the total space of the tangent bundle $TS^n$ of the $n$-sphere $S^n$ (e.g. \cite{Dimca}, p. 81). By considering the closure $\bar{F}_{\theta}$ as the closed tangent unit disk bundle $\mathcal{D}(TS^n)$ over $S^n$, we can identify the {\em binding} $K$ of $\mathcal{OB}$ as the tangent unit sphere bundle $\mathcal{S}(TS^n)$ over $S^n$. For our purposes we identify $TS^n$ with the cotangent bundle $T^*S^n$ by using the natural duality, and assume that each page $\bar{F}_{\theta}$  of $\mathcal{OB}$ is diffeomorphic to the cotangent unit disk bundle $\mathcal{D}(T^*S^n)$, and so the binding $K$ is diffeomorphic to the cotangent unit sphere bundle $\mathcal{S}(T^*S^n)$.

Now we define the function $\Pi:\mathbb{B}_{\varepsilon}^{2n+2} \to D^2$ from $(2n+2)$-ball $$\mathbb{B}_{\varepsilon}^{2n+2}=\{|z_1|^2+|z_2|^2+ \cdots +|z_{n+1}|^2\leq\varepsilon^2\} \subset \mathbb{C}^{n+1}$$ onto the unit disk $D^2 \subset \mathbb{C}$ by restricting $P$ and then normalizing by $\varepsilon^2$, that is, $$\Pi=\frac{1}{\varepsilon^2}P|_{\mathbb{B}_{\varepsilon}^{2n+2}}.$$
By definition (see \cite{Kas}, for instance) $\Pi$ is a local model for a Lefschetz fibration over the unit disk having only one singular fiber over the origin. Also \emph{regular} fibers $\Pi^{-1}(z), z\neq0$, are diffeomorphic to $\mathcal{D}(T^*S^n)$ because $\Pi$ is a topological locally trivial fibration on $D^2 \setminus \{0\}$ \cite{Le} (e.g. \cite{Dimca}, Chapter 3). Therefore, $\Pi$ defines a Lefschetz fibration $\mathcal{LF}$ on $\mathbb{B}_{\varepsilon}^{2n+2}$ which induces the open book $\mathcal{OB}$ on the boundary sphere $\mathbb{S}_{\varepsilon}^{2n+1}$. By definition, the monodromy of $\mathcal{LF}$ is the monodromy of $\mathcal{OB}$. According to \cite{Deligne_Katz,Kas} this monodromy is (up to isotopy) equal to the right-handed Dehn twist $$\delta: \mathcal{D}(T^*S^n) \rightarrow \mathcal{D}(T^*S^n)$$ along the \emph{vanishing cycle} which is the zero section (a copy of $S^n$) in $\mathcal{D}(T^*S^n)$. To describe $\delta$ precisely, identify the interior of a page with $T^*S^n$ and write the points in $T^*S^n$ as $(\textbf{q},\textbf{p}) \in \mathbb{R}^{n+1} \times \mathbb{R}^{n+1}$ such that $|\textbf{q}|=1$ and $\textbf{q} \perp \textbf{p}$. Then
$$\delta (\textbf{q}, \textbf{p})=
\begin{pmatrix} \cos g(|\textbf{p}|) &  |\textbf{p}|^{-1} \sin g(|\textbf{p}|) \\
                -|\textbf{p}| \sin g(|\textbf{p}|) & \cos g(|\textbf{p}|) \end{pmatrix}
\begin{pmatrix} \textbf{q} \\ \textbf{p}\end{pmatrix}$$
where $g$ is a smooth function that increases monotonically from $\pi$ to $2\pi$ on some interval, and outside this interval $g$ is identically equal to $\pi$ or $2\pi$. Observe that $\delta$ is the antipodal map on the zero section $S^n \times \{\textbf{0}\}=\{(\textbf{q},\textbf{p}) \,|\, |\textbf{q}|=1, \textbf{p}=\textbf{0}\}$, while it is the identity map for $|\textbf{p}|$ large. Note that as abstract open book $\mathcal{OB}$ is determined by the pair $(\mathcal{D}(T^*S^n),\delta)$.

Now let $z_j=x_j+iy_j$ for $j=1,...,n+1$. Then with respect to the complex coordinates $z=(z_1,...,z_{n+1})$, the \emph{standard Stein structure} on $\mathbb{C}^{n+1}$ (and hence on $\mathbb{B}_{\varepsilon}^{2n+2}$) is defined by the pair $$(J_0,\psi_0)=(i \times \cdots \times i, |\,z|\,^2).$$  This defines the \emph{standard symplectic} (indeed \emph{K\"ahler}) \emph{form}  $$\omega_0=-d(d\psi_0 \circ J_0)=\sum_{j=1}^{n+1}dx_j \wedge dy_j$$
whose \emph{Liouville vector field} $\chi_0 $ (i.e., satisfying $\mathcal{L}_{\chi_0}\omega_0=\omega_0$) is given by $$\chi_0= \frac{1}{2}\sum_{j=1}^{n+1} (x_j \,\partial x_j+y_j \,\partial y_j).$$
Then on the boundary sphere $\mathbb{S}_{\varepsilon}^{2n+1}$, the $1$-form $$\alpha_0=\iota_{\chi_0}\omega_0=\frac{1}{2}\sum_{j=1}^{n+1}(x_j \,dy_j-y_j \,dx_j)=\frac{i}{4}\sum_{j=1}^{n+1}(z_j \,d\bar{z}_j-\bar{z}_j \,dz_j)$$ is a \emph{contact form} (i.e., $\alpha_0 \wedge (d\alpha_0)^{\wedge n}|_{\,\mathbb{S}_{\varepsilon}^{2n+1}}>0$). The codimension one plane distribution  kernel $\xi_0=\textrm{Ker} (\alpha_0)$ is called the \emph{standard contact structure} on $\mathbb{S}_{\varepsilon}^{2n+1}$.

The compatibility between contact structures and open books is defined as follows:

\begin{definition} [\cite{Giroux}] \label{def:Compatibility}
We say that a contact structure $\xi = \textrm{Ker}(\alpha)$ on $M$ is
\emph{carried by} (or \emph{supported by}) an open book $(B, f)$ on $M$ (where $B$ is the binding), if the following conditions hold:

\begin{itemize}
\item[(i)] $(B, \alpha |_{TB})$ is a contact manifold.
\item[(ii)] For every $t\in S^1$, the
page $X=f^{-1}(t)$ is a symplectic manifold with symplectic
form $d\alpha$.
\item[(iii)] If $\bar{X}$ denotes the closure of a page $X$ in $M$, then the orientation of
$B$ induced by its contact form $\alpha |_{TB}$ coincides with its
orientation as the boundary of $(\bar{X}, d\alpha)$.
\end{itemize}
\end{definition}

The open book $\mathcal{OB}$ has been studied before, but since it is one of the main building blocks of the present paper and for completeness, here we discuss its important aspect:

\begin{lemma} \label{lem:Compatibility}
The open book $\mathcal{OB}$ carries the standard contact structure $\xi_0$ on $\mathbb{S}_{\varepsilon}^{2n+1}$ inherited from the standard Stein structure on $\mathbb{B}_{\varepsilon}^{2n+2}$.
\end{lemma}

\begin{proof}
The first condition of compatibility (Definition \ref{def:Compatibility}) immediately follows from \cite{Lutz_Meckert} where they show that the restriction of $\alpha_0$ is a contact form on Brieskorn  manifolds, and so, in particular, on the binding $K$. To check the second one, consider the vector field
$$R= \frac{4i}{\varepsilon^2}\sum_{j=1}^{n+1} z_j \,\partial z_j=R_0+R_1;\;  R_0=\frac{2i}{\varepsilon^2}\sum_{j=1}^{n+1}( z_j \,\partial z_j - \bar{z}_j \,\partial \bar{z}_j ), R_1=\frac{2i}{\varepsilon^2}\sum_{j=1}^{n+1}( z_j \,\partial z_j + \bar{z}_j \,\partial \bar{z}_j ).$$
Observe that $R_0|_{\mathbb{S}_{\varepsilon}^{2n+1}}$ is the \emph{Reeb vector field} of the contact form $\alpha_0|_{\mathbb{S}_{\varepsilon}^{2n+1}}$. (That is, we have $\alpha(R_0)=1, \iota_{R_0} d\alpha_0=0$ on $\mathbb{S}_{\varepsilon}^{2n+1}$.)
The flow of $R$ is computed as $$h_t(z)=(e^{4it/\varepsilon^2} z_1,..., e^{4it/\varepsilon^2} z_{n+1})$$ which is a $1$-parameter group of diffeomorphisms $h_t:\mathbb{C}^{n+1} \setminus Z(P) \to \mathbb{C}^{n+1} \setminus Z(P)$. Now consider the fibration $\Psi: \mathbb{C}^{n+1} \setminus Z(P) \to S^1$ given by $$\Psi(z)=\frac{P(z)}{|P(z)|}.$$ Then $h_t$ maps each fiber $\Psi^{-1}(y)$ diffeomorphically onto the fiber $\Psi^{-1}(e^{i\theta}y)$, and also there is a diffeomorphism $\Psi^{-1}(y) \cong \Theta^{-1}(y) \times \mathbb{R}$ as shown in Chapter 9 of \cite{Milnor}. Furthermore, $h_t$ maps $\mathbb{S}_{\varepsilon}^{2n+1} \setminus K$ diffeomorphically onto itself for all $t$. Hence, we conclude that $h_t$ maps each fiber $\Theta^{-1}(y)$ diffeomorphically onto the fiber $\Theta^{-1}(e^{i\theta}y)$, but this means, in particular, that the Reeb vector field $R_0|_{\mathbb{S}_{\varepsilon}^{2n+1}}$ is transverse to every page of the open book $\mathcal{OB}$ (note that $R_1$ does not live in $T\mathbb{S}_{\varepsilon}^{2n+1}$). So for any page $F_{\theta}$, the rank of $d\alpha_0 |_{F_{\theta}}$ is maximal which is equivalent to saying that $d\alpha_0$ is a symplectic form on $F_{\theta}$.

For the third condition, on $T\mathbb{B}_{\varepsilon}^{2n+2}|_{\mathbb{S}_{\varepsilon}^{2n+1}}$ we compute $\omega_0(\chi_0,R_0)=1$, so $\{\chi_0,R_0\}$ form a non-degenerate pair with respect to $\omega_0=d\alpha_0$. Therefore, for a fixed page $$F_{\theta}=\{z\, | \, \Theta(z)=e^{i\theta} \}  \subset \mathbb{S}_{\varepsilon}^{2n+1} \setminus K$$
of the fibration $\Theta$, the tangent bundle $T\mathbb{B}_{\varepsilon}^{2n+2}$ restricted to $F_{\theta}$ is decomposed as $$T\mathbb{B}_{\varepsilon}^{2n+2}|_{F_{\theta}}=\langle \chi_0 \rangle \oplus T\mathbb{S}_{\varepsilon}^{2n+1}|_{F_{\theta}}=\langle \chi_0 \rangle \oplus \langle R_0 \rangle \oplus T F_{\theta}$$ where we use the fact that the Liouville vector field $\chi_0$ is transverse to $\mathbb{S}_{\varepsilon}^{2n+1}$ (which is of contact type) and that $R_0$ is transverse to $F_{\theta}$. This shows that $(F_{\theta},\omega_0|_{F_{\theta}})$ is a symplectic submanifold of $(\mathbb{B}_{\varepsilon}^{2n+2},\omega_0)$ and, in particular, that the orientation on $F_{\theta}$ given by $\omega_0 |_{F_\theta}$ is inherited from the orientation on $\mathbb{B}_{\varepsilon}^{2n+2}$ given by $\omega_0$.

Write $\alpha'=\alpha_0|_{K}$, $F=F_{\theta}$. To finish the proof, we need to check that the orientation on $ \partial F=K$ given by the form $\alpha' \wedge (d\alpha')^{\wedge n-1}$ coincides with the one induced by the orientation on $F$ given by the volume form $(d\alpha_0|_F)^{\wedge n}$: We showed above that the latter orientation on $F$ is inherited from the one on $\mathbb{B}_{\varepsilon}^{2n+2}$ given by the standard Stein structure $(J_0,\psi_0)$. Note that the orientation on $(\mathbb{S}_{\varepsilon}^{2n+2},\xi_0)$ given by the volume form $\alpha_0 \wedge (d\alpha_0)^{\wedge n}$ is also coming from this Stein structure. Moreover, the orientation on $(K,\xi':=\textrm{Ker} (\alpha')) \subset (\mathbb{S}_{\varepsilon}^{2n+1},\xi_0)$ determined by $\alpha' \wedge (d\alpha')^{\wedge n-1}$ matches up with the one inherited (as a contact submanifold) from $(\mathbb{S}_{\varepsilon}^{2n+1},\xi_0)$. Hence, the mentioned two orientations on $K$ must coincide.
\end{proof}

\subsection{Positive stabilization of open books}

We first recall the plumbing or 2-Murasugi sum of two \emph{contact open books} (i.e., open books carrying contact structures): Let $(M_i, \xi_i)$ be two closed contact manifolds such that each $\xi_i$ is carried by an open book $(\Sigma_i,h_i)$ on $M_i$. Suppose that $L_i$ is a properly embedded Lagrangian ball in $\Sigma_i$ with Legendrian boundary $\partial L_i \subset \partial \Sigma_i$. By the Weinstein neighborhood theorem each $L_i$ has a standard neighborhood $N_i$ in $\Sigma_i$ which is symplectomorphic to $(T^* D^n, d\lambda_{\textrm{can}})$ where $\lambda_{\textrm{can}}=\textbf{p}\textbf{d}\textbf{q}$ is the canonical $1$-form on $\mathbb{R}^n \times \mathbb{R}^n$ with coordinates $(\textbf{q},\textbf{p})$. Then the \emph{plumbing}
or \emph{2-Murasugi sum} $(\mathcal{P}(\Sigma_1,\Sigma_2;L_1,L_2),h)$ of $(\Sigma_1,h_1)$ and $(\Sigma_2,h_2)$ \emph{along} $L_1$ and $L_2$ is the open book on the connected sum $M_1 \# M_2$ with the pages obtained by gluing $\Sigma_i$'s together along $N_i$'s by interchanging $\textbf{q}$-coordinates in $N_1$ with $\textbf{p}$-coordinates in $N_2$, and vice versa. To define $h$, extend each $h_i$ to $\tilde{h}_i$ on the new page by requiring $\tilde{h}_i$ to be identity map outside the domain of $h_i$. Then the monodrodmy $h$ is defined to be $\tilde{h}_2 \circ \tilde{h}_1$. Without abuse of notation we will drop the ``tilde'' sign, and write $h=h_2 \circ h_1$.

The following terminology was given in \cite{Giroux}. We describe it in a slightly different way so that it fits into the notation of the present paper.

\begin{definition}[\cite{Giroux}] \label{def:Stabilization_Contac_Open Book}
Suppose that $(\Sigma,h)$ carries the contact structure $\xi=\textrm{Ker}(\alpha)$ on a $(2n+1)$-manifold $M$. Let $L$ be a properly embedded Lagrangian $n$-ball in a page $(\Sigma,d\alpha)$
such that  $\partial L \subset \partial \Sigma$ is a Legendrian
$(n-1)$-sphere in the binding $(\partial \Sigma,\alpha |_{\,\partial \Sigma})$.
Then the \emph{positive} (or \emph{standard}) \emph{stabilization} $\mathcal{S_{OB}}[(\Sigma,h);L]$ of $(\Sigma,h)$
\emph{along} $L$ is the open book $(\mathcal{P}(\Sigma,\mathcal{D}(T^*S^n);L,\textbf{D}), \delta \circ h)$ where $\textbf{D}\cong D^n$ is any fiber in $\mathcal{D}(T^*S^n)$.
\end{definition}

\subsection{Characterization of Lefschetz fibrations}
Here we recall the handle decomposition of Lefschetz fibrations as described in \cite{Kas}: Let $\pi: W \rightarrow D^2\subset \C$ be a given Lefschetz fibration with a regular fiber $X^{2n}$ and monodromy $h$. Consider the base disk as $D^2= \{z \in \C : \left|z\right| \leq 2 \}$. We may assume that $0 \in D^2$ and the points on $\partial D^2$ are regular values and that all the critical values $\{\lambda_1,\lambda_2,..,\lambda_{\mu}\}$ of $\pi$ are $\mu$ roots of unity. Such a $\pi$ is called a \emph{normalized Lefschetz fibration}. Since every Lefschetz fibration can be normalized, throughout the paper all Lefschetz fibrations will be assumed to be normalized. Define a Morse function $F:W \rightarrow [0,4] \subset \R$ given by $F(x)= {\left| \pi(x) \right|}^2$. Then outside of the set $F^{-1}(0) \cup F^{-1}(4)$, $F$ has only nondegenerate critical points of index $n+1$ (see \cite{Andreotti_Frankel}). Since $\left|\lambda_i\right|=1$ for all $i$, the map $\pi$ has no critical values on the set $D_t= \{z \in \C : \left|z\right| \leq t \}$ for $t<1$ and hence

\begin{center}
 $F^{-1}([0,t])=\pi^{-1}(D_t) \cong X \times D^2 \quad $ for $t<1$.
\end{center}

On the other hand, for $t>1$, $\pi^{-1}(D_t)$ is diffeomorphic to the manifold
obtained from $X \times D^2$ by attaching $\mu$ handles of
index $n+1$, via the attaching maps $$\Phi_j: S^n \times D^{n+1}
\rightarrow \partial(X \times D^2)= X \times S^1, \quad j=1, 2,...,
\mu.$$

Let $\Phi^{\prime}_j : \epsilon^{n+1} \rightarrow \nu$
be the framing of the $j$-th handle, where $\epsilon^k$ denotes the
trivial bundle of rank $k$, and $\nu$ denotes the normal bundle of the attaching sphere
$\Phi_j(S^n \times \{0\})$ in $\partial(X \times D^2)$.

\begin{fact} [\cite{Kas}] The embeddings $\Phi_j$ may be chosen so that for
each $j=1,2,...,\mu$ there exists $z_j$ such that
$ \Phi_j(S^n \times \{0\}) \subset \pi^{-1}(z_j)\cong X $.
\end{fact}

So, set $\phi_j: S^n \rightarrow X$ to be the embedding defined
by restricting $\Phi_j$ to $S^n \times \{0 \}$. Let $\nu_1$ denote
the normal bundle of $S^n \cong \phi_j(S^n)$ in $X$ corresponding
to the embedding $\phi_j$, and consider $\nu$ as the normal bundle
of $\phi_j(S^n)$ in $F^{-1}(1-\delta)$. Clearly, $\nu \cong \nu_1 \oplus \epsilon$ (as the normal bundle of $X$ in $W$ is trivial). Let $\tau$ denote the tangent bundle of $S^n$.

\begin{fact} [\cite{Kas}] For each $j=1,2,...,\mu$, there exists a bundle
isomorphism $\phi^{\prime}_j : \tau\rightarrow \nu_1$ such that
the framing $\Phi^{\prime}_j$ of the $(n+1)$-handle corresponding
to $\lambda_j$ coincides with $\phi^{\prime}_j$. That is,
$\Phi^{\prime}_j$ is given by the composition
$$\epsilon^{n+1}\stackrel{\equiv}{\longrightarrow} \tau \oplus
\epsilon \stackrel{\phi^{\prime}_j\oplus id}{\longrightarrow} \nu_1
\oplus \epsilon \stackrel{\equiv}{\longrightarrow} \nu .$$
\end{fact}

\begin{definition} [\cite{Kas}]\label{def:Normalized_vanishing_cycle}
$S^n\cong \phi_j(S^n)$ is called a \textit{vanishing cycle} of
$\pi$. The bundle isomorphism $\phi^{\prime}_j : \tau\rightarrow
\nu_1$ is called a \textit{normalization} of $\phi_j$. The pair
$(\phi_j,\phi'_j)$ is called a \textit{normalized
vanishing cycle}.
\end{definition}

Let $\mathcal{D}(TS^{n}) \subset \tau$ denote the closed tangent unit disk bundle of $S^n$. By the tubular neighborhood
theorem and the canonical isomorphism $\mathcal{D}(T^*S^{n})\cong \mathcal{D}(TS^{n})$, we can apply the right-handed Dehn twist $\delta$ to a tubular neighborhood of $\phi_j(S^n)$ in $X$, and we can extend $\delta$, by the identity, to a self-diffeomorphism of $X$
which we denote by $$\delta_{(\phi_j,\phi_j')}:X \stackrel{\approx}{\longrightarrow} X.$$

Up to smooth isotopy $\delta_{(\phi_j,\phi_j')}\in \textrm{Diff}(X)$ depends only on the
smooth isotopy class of the embedding $\phi_j$ and the bundle isomorphism $\phi_j'$.

\begin{definition} [\cite{Kas}] 
$\delta_{(\phi_j,\phi_j')}$ is called the \emph{right-handed Dehn twist with center} $(\phi_j,\phi_j')$.
\end{definition}

We will make use of the following theorem.

\begin{theorem}[\cite{Kas}, \cite{Deligne_Katz}] \label{thm:Determination_Lefschetz_fibration}
The Lefschetz fibration $\pi: W \rightarrow D^2$ is uniquely determined by a sequence of vanishing cycles $(\phi_1, \phi_2,..., \phi_{\mu})$ and a sequence of their normalizations $(\phi^{\prime}_1,
\phi^{\prime}_2,..., \phi^{\prime}_{\mu})$. The monodromy of the fibration is equal to $$\delta_{\mu} \circ \cdots \circ \delta_2 \circ \delta_1 \in \emph{Diff}(X)$$ where $\delta_j=\delta_{(\phi_j,\phi_j')}$ is the right-handed Dehn twist with center $(\phi_j,\phi_j')$.
\end{theorem}

\begin{remark} \label{rmk:Dehn_twist_Symplectomorphism} Recall the right-handed Dehn twist $\delta:\mathcal{D}(T^*S^n) \to \mathcal{D}(T^*S^n)$ given explicitly before. With respect to the coordinates $(\textbf{q}, \textbf{p})$ on $\mathbb{R}^{2(n+1)}$ consider the canonical $1$-form $\lambda_{can} = \textbf{p}\cdot d\textbf{q}$ on $\mathcal{D}(T^*S^n) \subset \mathbb{R}^{2(n+1)}$. Then one can compute $$\delta^* \lambda_{can} = \lambda_{can} + |\textbf{p}| d(g(|\textbf{p}|))$$
which implies that the difference $\delta^* \lambda_{can}-\lambda_{can}$ is exact. Therefore, $\delta$ is a symplectomorphism of the symplectic manifold $(\mathcal{D}(T^*S^n),d\lambda_{can})$. As a result, if a regular fiber $X$ of a Lefschetz fibration $\pi: W \rightarrow D^2$ equipped with a symplectic structure $\omega$, then the monodromy $h$ of $\pi$ is a symplectomorphism of $(X,\omega)$. That is, $$h=\delta_{\mu} \circ \cdots \circ \delta_2 \circ \delta_1 \in \textrm{Symp}(X,\omega)$$ where $\delta_j=\delta_{(\phi_j,\phi_j')}$ is the right-handed Dehn twist with center $(\phi_j,\phi_j')$ as in Theorem \ref{thm:Determination_Lefschetz_fibration}.
\end{remark}

\begin{notation} \label{not:Smooth_Lefschetz_Fibration}
For our purposes it is convenient to define a notation for Lefschetz fibrations. Let the quadruple  $(\pi,W,X,h)$ denote the Lefschetz fibration $\pi:W \to D^2$ on $W$ with a regular fiber $X$ and the monodromy $h$. For instance, according to this notation we have $\mathcal{LF}=(\Pi,\mathbb{B}_{\varepsilon}^{2n+2}, \mathcal{D}(T^*S^n),\delta)$.
\end{notation}

For completeness we give the following basic well-known fact as a definition:

\begin{definition} \label{def:Induced_Openbook}
Let $(\pi,W,X,h)$ be any (normalized) Lefschetz fibration. The pairs $(\partial \pi^{-1}(0), \pi|\,_{\partial W})$ and $(X,h)$ are both called the \emph{induced open book} (or sometimes the \emph{boundary open book}) on $\partial W$.
\end{definition}


\section{Positive stabilization of Lefschetz fibrations} \label{sec:Positive stabilization of Lefschetz fibrations}
Now we define a process on Lefschetz fibrations as a counterpart of positive stabilization on open books. We will use Weinstein handles introduced in \cite{Weinstein}. Using the symplectization model near convex boundaries, these handles can be glued to symplectic manifolds along isotropic spheres to obtain new ones, and they give elementary symplectic cobordisms between contact manifolds. We will briefly explain them later.

\begin{definition} \label{def:Stabilization_Lefschetz_Fibration}
Let $(\pi,W,X,h)$ be a Lefschetz fibration which induces a contact open book on $\partial W$. Suppose that $L \subset (X,\omega)$ is a properly embedded Lagrangian $n$-ball with a Legendrian boundary $\partial L \subset \partial X$ on a page of the induced open book. Then the \emph{positive stabilization} $\mathcal{S_{LF}}[(\pi,W,X,h);L]$ of $(\pi,W,X,h)$ \emph{along} $L$ is a Lefschetz fibration $(\pi',W',X',h')$ described as follows:\\

\begin{itemize}
\item[(I)] $X'$ is obtained from $X$ by attaching a Weinstein $n$-handle
$H=D^n\times D^n$ along the Legendrian sphere $\partial L \subset \partial X$.\\
\item[(II)] $h'=\delta_{(\phi,\phi')} \circ h$ where $\delta_{(\phi,\phi')}$ is the
right-handed Dehn twist with center $(\phi,\phi')$ defined as follows: $\phi(S^n)$
is the Lagrangian $n$-sphere $S=D^n \times \{0\} \cup_{\partial L} L$ in the
symplectic manifold $(X'=X \cup H, \omega')$ where $\omega'$ is obtained by
gluing $\omega$ and standard symplectic form on $H$. If $\nu_1$ denote the normal bundle of $S$
in $X'$, then the normalization $\phi': \tau \rightarrow \nu_1$ is given by the
bundle isomorphisms $$\tau \underset{\phi_*}{\stackrel{\cong}{\longrightarrow}} TS \stackrel{\cong}{\longrightarrow} TX'/TS = \nu_1.$$
\end{itemize}
\end{definition}

\begin{remark} $W'$ is, indeed, diffeomorphic to $W$ (see the proof of Theorem \ref{thm:openbooks_Lefs.fibs} below). Also in $h'=\delta_{(\phi,\phi')} \circ h$, we think of $h$ as an element in $\textrm{Diff}\,(X')$ by trivially extending over $H$. Moreover, the isomorphism $TS \rightarrow TX'/TS$ exists because $S$ is Lagrangian in $(X',\omega')$ (the core $D^n \times \{0\}$ of $H$ is Lagrangian). Finally, note that there is a strong analogy between $\mathcal{S_{OB}}[(X,h);L]$ and $\mathcal{S_{LF}}[(\pi,W,X,h);L]$. On the one hand, we have $$\mathcal{S_{OB}}[(X,h);L]=(\mathcal{P}(X,\mathcal{D}(T^*S^n);L,\textbf{D}), \delta \circ h)$$
which means that we are plumbing the open book $(X,h)$ on a given manifold $M$ with the open book $\mathcal{OB}=(\mathcal{D}(T^*S^n),\delta)$ on $\mathbb{S}_{\varepsilon}^{2n+1}$. Therefore, $\mathcal{S_{OB}}[(X,h);L]$ is an open book on the connected sum $M \#\, \mathbb{S}_{\varepsilon}^{2n+1} \approx M$.  On the other hand, we may regard $\mathcal{S_{LF}}[(\pi,W,\omega,\chi,X,h);L]$ as the result of (informally speaking) ``Lefschetz plumbing'' of $(\pi,W,X,h)$ with $\mathcal{LF}$. Indeed, one can see that $\mathcal{S_{LF}}[(\pi,W,X,h);L]$ is a Lefschetz fibration on the boundary connect sum $W \#_{b}\, \mathbb{B}_{\varepsilon}^{2n+2} \approx W$.
\end{remark}

We are now ready to prove

\begin{theorem} \label{thm:openbooks_Lefs.fibs}
Any positive stabilization $\mathcal{S_{LF}}[(\pi,W,X,h);L]$ of a Lefschetz fibration $(\pi,W,X,h)$ with a contact boundary open book induces the open book $\mathcal{S_{OB}}[(X,h);L]$.  Conversely, if a contact open book $(X,h)$ is induced by a Lefschetz fibration $(\pi,W,X,h)$, then any positive stabilization $\mathcal{S_{OB}}[(X,h);L]$ of $(X,h)$ is induced by $\mathcal{S_{LF}}[(\pi,W,X,h);L]$.
\end{theorem}

\begin{proof} By definition of $\mathcal{S_{LF}}$, the fiber $X'$ is obtained from $X$ by attaching $2n$-dimensional Weinstein $n$-handle $H$ along $\partial L \subset \partial X$ . Since every fiber over $D^2$ is gaining $H$, we are actually attaching a $(2n+2)$-dimensional handle $$H'=H \times D^2=D^n \times D^{n+2}$$ to $W$ along $\partial L \subset \partial W$. Say the resulting manifold is $\widetilde{W}$, that is $\widetilde{W}=W \cup H'$. By extending the monodromy $h$ (but calling it still $h$) trivially over $H$, we get an extended Lefschetz fibration $\widetilde{\pi}: \widetilde{W} \to D^2$ on $\widetilde{W}$, i.e., we get $(\widetilde{\pi}, \widetilde{W},X',h)$. Note that $(\widetilde{\pi}, \widetilde{W},X',h)$ is determined by Theorem \ref{thm:Determination_Lefschetz_fibration}. So far what we explained is the content of Stage (I) in Definition \ref{def:Stabilization_Lefschetz_Fibration}. In Stage (II), composing the monodromy $h$ with $\delta_{(\phi,\phi')}$ corresponds to attaching an $(2n+2)$-dimensional handle $H''$ (so called a ``Lefschetz handle'') with index $n+1$ to $\widetilde{W}$ along the Lagrangian sphere $S$ in the fiber $(X',\omega')$ of $(\widetilde{\pi}, \widetilde{W},X',h)$. By Theorem \ref{thm:Determination_Lefschetz_fibration}, we know that $(\widetilde{\pi}, \widetilde{W},X',h)$ extends over the handle $H''$ and we get the Lefschetz fibration $\mathcal{S_{LF}}[(\pi,W,X,h);L]=(\pi',W',X',h')$ on the resulting manifold $$W'=\widetilde{W} \cup H''=W \cup H' \cup H''.$$ We immediately see that $\{H',H''\}$ form a canceling pair in the smooth category as the attaching sphere of $H''$ intersects the belt sphere of $H'$ transversely once, and so $W'$ is  diffeomorphic to the original manifold $W$ (indeed $W'=W \#_{b} \mathbb{B}_{\varepsilon}^{2n+2}$). Therefore, the open book $(X',h')$ induced by $\mathcal{S_{LF}}[(\pi,W,X,h);L]$ is an open on the original boundary $\partial W$. Next we need to see that $(X',h')$ is indeed isomorphic (as an abstract open book) to $\mathcal{S_{OB}}[(X,h);L]$. To this end, first observe that in the plumbing $(\mathcal{P}(X,\mathcal{D}(T^*S^n);L,\textbf{D}), \delta \circ h)$ we are embedding a tubular neighborhood $N\mathbf{(D)}$ of $\textbf{D}$ in $\mathcal{D}(T^*S^n)$ into the page $X$ in such a way that the intersection $N\mathbf{(D)} \cap \partial X$ is a tubular neighborhood of the Legendrian sphere $\partial L(\approx S^{n-1})$. Considering $\partial L$ as the equator of the zero section $S^n\times \{\textbf{0}\}\subset \mathcal{D}(T^*S^n)$, it is clear that the part  $\mathcal{D}(T^*S^n) \setminus N\mathbf{(D)}$ of $\mathcal{D}(T^*S^n)$ which is not mapped into $X$ (during the plumbing) is the trivial bundle $\mathcal{D}(T^*D^n)\cong D^n \times D^n$. Note that the canonical symplectic structure on $\mathcal{D}(T^*S^n)$ restricts to the standard symplectic structure on $\mathcal{D}(T^*D^n)$ which implies that $\mathcal{D}(T^*S^n) \setminus N\mathbf{(D)}$ is the Weinstein handle $H$ glued to $X$ along $\partial L$. Hence, the page of the open book $\mathcal{S_{OB}}[(X,h);L]$, that is, the resulting page of the plumbing, is $X'$. Also if we keep track of the vanishing cycle $S^n\times \{\textbf{0}\} \subset \mathcal{D}(T^*S^n)$ in the above discussion, we immediately see that it corresponds to the Lagrangian $n$-sphere $S=C \cup_{\partial L} L$ where $C=D^n \times \{0\}$ is the (Lagrangian) core disk of the Weinstein handle $H$ which means that the right-handed Dehn twist $\delta$ coincides with $\delta_{\phi,\phi'}$ described in Definition \ref{def:Stabilization_Lefschetz_Fibration}. Composing with $h$, we get $\delta \circ h=h'$ Thus, $\mathcal{S_{OB}}[(X,h);L]$ and $(X',h')$ are isomorphic. This proves the first statement.\\

For the second statement we basically follow the same steps in a different order: If $\mathcal{S_{OB}}[(X,h);L]$ is a given stabilization, then by the above discussion we know that the new page is equal to $X'=X \cup H$. By assumption $(X,h)$ is induced from $(\pi,W,X,h)$. So by attaching $H'=H \times D^2$ (thickening of $H$) to $W$, each fiber of $\pi$ gains the handle $H$, and we get $(\widetilde{\pi}, \widetilde{W},X',h)$ on $\widetilde{W}$. Since $\delta=\delta_{(\phi,\phi')}$, $h'=\delta_{(\phi,\phi')} \circ h=\delta \circ h$. Therefore, we have $\mathcal{S_{OB}}[(X,h);L]=(X',h')$. Moreover, composing $\delta$ with $h$ (in the open book level) corresponds to attaching a Lefschetz handle $H''$ to $\widetilde{W}$ whose normalized vanishing cycle is $(\phi,\phi')$. Therefore, we obtain $(\pi',W',X',h')$ on $W'=\widetilde{W} \cup H''(\approx W)$. It is now clear that $\mathcal{S_{OB}}[(X,h);L]$ is induced by $\mathcal{S_{LF}}[(\pi,W,X,h);L]$.
\end{proof}


\section{Exact Open Books} \label{sec:Exact_Open_Books}
We will define exact open books as boundary open books induced by exact Lefschetz fibrations. To this end, recall that a contact manifold $(M,\alpha)$ is called \emph{strongly symplectically filled} by a symplectic manifold $(X,\omega)$ if there exist a Liouville vector field $\chi$ of $\omega$ defined (at least) locally near $\partial X=M$ such that $\chi$ is transverse to $M$ and $\iota_{\chi}\omega=\alpha$. Such a boundary is called \emph{convex}. An \emph{exact symplectic manifold} is a compact manifold $X$ with boundary, together with a symplectic form $\omega$ and a $1$-form $\alpha$ satisfying $\omega=d\alpha$, such that $\alpha|\,_{\partial X}$ is a contact form which makes $\partial X$ convex. In such a case there is a Liouville vector field $\chi$ of $\omega$ such that $\iota_{\chi} \omega=\alpha$. We will write exact symplectic manifolds as triples $(X,\omega,\alpha)$ (or sometimes as quadruples $(X,\omega,\alpha, \chi)$). Also the pair $(\omega,\alpha)$ (or sometimes the triple $(\omega,\alpha, \chi)$) will be called an \emph{exact symplectic structure} on $X$.

Let $\pi:E^{2n+2} \to S$ be a differentiable fiber bundle, denoted by $(\pi,E)$, whose fibers and base
are compact connected manifolds with boundary. The boundary of such an $E$ consists of two parts: The vertical part $\partial_v E;=\pi^{-1}(\partial S)$, and the horizontal part $\partial_h E:=\bigcup_{z \in S} \partial E_z$ where $E_z=\pi^{-1}(z)$ is the fiber over $z \in S$. The following definitions can be found in \cite{S1,Seidel}.

\begin{definition} [\cite{S1,Seidel}] \label{def:Exact_Symplectic_Fibration}
An \emph{exact symplectic fibration} $(\pi,E,\omega,\alpha)$ over a bordered surface $S$ is a differentiable fiber bundle $(\pi,E)$ equipped with a $2$-form $\omega$ and a $1$-form $\alpha$ on $E$, satisfying $\omega=d\alpha$, such that
\begin{itemize}
\item[(i)] each fiber $E_z$ with $\omega_z=\omega|_{E_z}$ and $\alpha_z=\alpha|_{E_z}$ is an exact symplectic manifold,
\item[(ii)] the following triviality condition near $\partial_h E$ is satisfied: Choose a point $z \in S$ and consider the trivial fibration $\tilde{\pi} : \tilde{E}:= S \times E_z \to S$ with the forms $\tilde{\omega}, \tilde{\alpha}$ which are pullbacks of  $\omega_z,\alpha_z$, respectively. Then there should be a fiber-preserving diffeomorphism $\Upsilon:N \to \tilde{N}$ between neighborhoods $N$ of $\partial_h E$ in $E$ and $\tilde{N}$ of $\partial_h \tilde{E}$ in $\tilde{E}$ which maps $\partial_h E$ to $\partial_h \tilde{E}$, equals the identity on $N \cap E_z$, and $\Upsilon ^*\tilde{\omega}=\omega$ and $\Upsilon ^* \tilde{\alpha}=\alpha$.
\end{itemize}
\end{definition}

\begin{definition} [\cite{S1,Seidel}] \label{def:Exact_Lefschetz_Fibration}
An \emph{exact Lefschetz fibration} is a tuple $(\pi,E,S,\omega,\alpha,J_0,j_0)$ which satisfies the following conditions:
\begin{itemize}
\item[(i)] $\pi:E \to S$ is allowed to have finitely many critical points all of which lie in the interior of $E$.
\item[(ii)] $\pi$ is injective on the set $C$ of its critical points.
\item[(iii)] $J_0$ is an integrable complex structure defined in a neighborhood of $C$ in $E$ such that $\omega$ is a K\"ahler form for $J_0$.
\item[(iv)] $j_0$ is a positively oriented complex structure on a neighborhood of the set $\pi(C)$ in $S$ of the critical values.
\item[(v)] $\pi$ is $(J_0,j_0)$-holomorphic near $C$.
\item[(vi)] The Hessian of $\pi$ at any critical point is nondegenerate as a complex quadratic form, in other words, $\pi$ has nondegenerate complex second derivative at each its critical point.
\item[(vii)]  $(\pi,E\setminus \pi^{-1}(\pi(C)),\omega,\alpha)$ is an exact symplectic fibration over $S \setminus \pi(C)$.
\end{itemize}
\end{definition}

\begin{remark} \label{rem:Exact_Lefschetz_Fibration}
As pointed out in \cite{Seidel}, one can find an almost complex structure $J$ on $E$ agreeing with $J_0$ near $C$ and a positively oriented complex structure $j$ on $S$ agreeing with $j_0$ near $\pi(C)$ such that $\pi$ is $(J,j)$-holomorphic and $\omega(\cdot,J \cdot)|_{\textrm{Ker}\,(\pi_*)}$ is symmetric and positive definite everywhere. The existence of $(J,j)$ is guaranteed by the fact that the space of such pairs $(J,j)$ is always contractible, and in particular, always nonempty. Furthermore, once we fixed $(J,j)$, we can modify $\omega$ by adding a positive $2$-form on $S$ so that it becomes symplectic and tames $J$ everywhere on $E$. Let $\Omega=\omega + \pi^*(\omega_S)$ be such a modification of $\omega$ where $\omega_S$ is a volume form on $S$ which is sufficiently positive. Since any volume form on a bordered surface is exact, $\omega_S=d\alpha_S$ for some $1$-form $\alpha_S$ on $S$, and so $\Omega=d\Lambda$ where $\Lambda=\alpha+\pi^*(\alpha_S).$ Therefore, it is natural to ask if  $(\Omega,\Lambda)$ defines an exact symplectic structure on $E$. For that one needs to check that $\Omega$ admits a Liouville vector field defined in a colar neighborhood of $\partial E$ which is transversely pointing out from $\partial E$. This observation suggests the following definition.
\end{remark}

\begin{definition} \label{def:compatible_exact_mfld_fib}
An exact symplectic manifold $(E,\Omega,\Lambda)$ and an exact Lefschetz fibration $(\pi,E,S,\omega,\alpha,J_0,j_0)$ are said to be \emph{compatible} if for a pair $(J,j)$ as in Remark \ref{rem:Exact_Lefschetz_Fibration} there exists a positive volume form $\omega_S$ on $S$ such that  $$\Omega=\omega + \pi^*(\omega_S)$$ and $\Omega$ tames $J$ everywhere on $E$.
\end{definition}

\begin{lemma} \label{lem:Exact_Lefschetz_Fibration}
For any compatible exact Lefschetz fibration $(\pi,E,S,\omega,\alpha,J_0,j_0)$ on an exact symplectic manifold $(E,\Omega,\Lambda,\chi)$, the following holds
\begin{itemize}
\item[(i)] There exist a neighborhood $N$ of $\partial_h E$ in $E$ on which the Liouville vector field $\chi$ of $\Omega$ has a decomposition $$\chi=V+H$$ where $V,H \in \Gamma(TN)$ such that $V$ is a section of $K:=\emph{Ker}(\pi_*)|_{N} \subset TN$ and $H$ is a section of the $\Omega$-symplectic complement $K^{\Omega}$ of  $K$ in $TN$.
\item[(ii)] Each regular fiber $(E_z,\omega_z,\alpha_z,\chi_z)$ is an exact symplectic submanifold of $(E,\Omega,\Lambda,\chi)$ where $\chi_z \subset \Gamma(T(E_z \cap N))$ is the Liouville vector field of $\omega_z$ transversally pointing out from $\partial E_z$.
\end{itemize}
\end{lemma}

\begin{proof}
By assumption we have $\Omega=\omega + \pi^*(\omega_S)$ where $\omega_S$ is a positive volume form on $S$ (equipped with a complex structure $j$ as above). 

For the first statement,  consider the Liouville vector field $\chi_z$ of $\omega_z$ on a given fiber $E_z$. By the local triviality condition near $\partial_h E$ (i.e., (ii) in Definition \ref{def:Exact_Symplectic_Fibration}), these $\chi_z$'s glue together (smoothly) and gives a Liouville vector field $V \in \Gamma(TN)$ for $\omega$ on some colar neighborhood $N$ of $\partial_h E$. Note that $V \in K$ as being a union of vertical vector fields. 
 
The complex structure $j$ and the positive volume (or symplectic) form $\omega_S$ determine a Riemannian, and so a Stein structure on $S$ (recall $S$ is connected). In particular, $\omega_S$ has a globally defined Liouville vector field $\chi_S$ transversally pointing out from $\partial S$. Since $\pi|_N:N \to S$ is a trivial fibration, we can lift $\chi_S$ to its obvious lift $H \in \Gamma(TN)$. The condition (ii) in Definition \ref{def:Exact_Symplectic_Fibration} also ensures that along the intersection of $N$ with the preimage of any local chart on $S$, the local description of $\omega$ involves only vertical (fiber) coordinates (Lemma 1.1 in \cite{Seidel}). Therefore, $\omega(u,H)=0, \forall u \in K$ which implies that $\Omega(u,H)=0, \forall u \in K$ (as $\pi^*(\omega_S)(u,H)=0, \forall u \in K$). As a result, we conclude that $H \in \Gamma(K^{\Omega})$.

Now for the vector field $V+H$ one easily checks that $$\mathcal{L}_{V+H}\Omega=\mathcal{L}_{V}\omega+\underbrace{\mathcal{L}_{H}\omega}_{=0}+
\underbrace{\mathcal{L}_{V}\pi^*(\omega_S)}_{=0}+\mathcal{L}_{H}\pi^*(\omega_S)=\omega+\pi^*(\omega_S)=\Omega.$$ Since $V+H$ is transversally pointing out from $\partial N \cap \partial E$, we conlude that it must coincide with the Liouville vector field $\chi$ of $\Omega$ restricted to $N$. 

For the second part, first observe that $\Omega|_{E_z}=\omega|_{E_z}=\omega_z$ and $\Lambda|_{E_z}=\alpha|_{E_z}=\alpha_z$. Also from part (i) we know that $\chi=V+H$ where $H$ has no vertical component. Moreover, $V$ was  constructed by taking the unions of $\chi_z$'s. Therefore, when restricted to $E_z \cap N$, the vertical (fiber) component of the Liouville vector field $\chi$ is equal to  $\chi_z$. Hence, each regular fiber $(E_z,\omega_z,\alpha_z,\chi_z)$ is an exact symplectic submanifold of the total space $(E,\Omega,\Lambda,\chi)$.
\end{proof}

\begin{conventionandnotation} \label{not:Exact_Lefschetz_Fibration}
From now on, the base of any exact Lefschetz fibration will be assumed to be the $2$-disk $D^2$. Since they will not be considered in our discussions, we drop $J_0$ and $j_0$ from our notation. As an another convension, any exact Lefschetz fibration will be assumed to be compatible with some exact symplectic structure on its total space. Moreover, we will assume that $\omega=d\alpha$ in Definition \ref{def:Exact_Lefschetz_Fibration}  has been already modified as in Remark \ref{rem:Exact_Lefschetz_Fibration} or in the proof of Lemma \ref{lem:Exact_Lefschetz_Fibration}, and will include its Liouville vector field $\chi$ in the notation. Furthermore, we also want to specify the regular fiber and the monodromy in our notation as before. Therefore, we introduce the following:

Let $(\pi,E,\omega,\alpha,\chi,X,h)$ denote a compatible exact Lefschetz fibration over the disk $D^2$ with the following properties:
\begin{itemize}
\item[(i)] The underlying smooth Lefschetz fibration is $(\pi,E,X,h)$ with $h \in \textrm{Symp}(X,\omega_X)$.
\item[(ii)] $(E,\omega,\alpha,\chi)$ is an exact symplectic manifold with convex boundary $(\partial E,\alpha|_{\partial E})$.
\item[(iii)] Each regular fiber $(E_z,\omega_z,\alpha_z,\chi_z)$ is an exact symplectic submanifold of $(E,\omega,\alpha,\chi)$.
\end{itemize}
Note that any compatible exact Lefschetz fibration over the disk $D^2$ admits such representation.
\end{conventionandnotation}

\begin{definition} \label{def:Exact_Open_Book}
If an open book is induced by a compatible exact Lefschetz fibration, then it is said to be an \emph{exact open book}.
\end{definition}

\begin{theorem} \label{thm:Exact_Openbooks_Support}
The exact open book $(X,h)$ induced by a compatible exact Lefschetz fibration $(\pi,E,\omega,\alpha,\chi,X,h)$ caries the contact structure $\xi=\emph{Ker}(\alpha|_{\,\partial E})$ on $\partial E$.
\end{theorem}

\begin{proof} We need to show that all three conditions in Definition \ref{def:Compatibility} hold. Assuming $(\pi,E,\omega,\alpha,\chi,X,h)$ is normalized (and so $z_0=(0,0)$ is a regular value), the binding of $(X,h)$ is the boundary of the regular fiber $E_{z_0}=\pi^{-1}(z_0)$. We know that $(\partial E_{z_0},\alpha_{z_0}|_{\partial E_{z_0}})$ is the convex boundary of $(E_{z_0},\omega_{z_0},\alpha_{z_0})$ which is an exact symplectic submanifold of $(E,\omega,\alpha)$.
Since $\alpha_{z_0}|_{\partial E_{z_0}}=(\alpha|_{E_{z_0}})|_{\partial E_{z_0}}=\alpha|_{\partial E_{z_0}}$, we conclude that $\alpha|_{\partial E_{z_0}}$ is a contact form on the binding $\partial E_{z_0}$, so the first condition follows.

For the second one, each regular fiber $E_z$ of $(\pi,E,\omega,\alpha,\chi,X,h)$ is an exact symplectic submanifold of $(E,\omega,\alpha)$ with the symplectic form $\omega_z=\omega|_{E_z}=d\alpha|_{E_z}$. In particular, any page $X$ of the boundary open book $(X,h)$ equips with the symplectic structure $d\alpha|_{X}$ as being a regular fiber of $\pi$.

To check the orientation condition, we need to specify the dimensions. Say $E$ has dimension $2n+2$, and so the page $X$ and the binding $B$ have dimensions $2n$ and $2n-1$, respectively. For simplicity,
write $\alpha'=\alpha|_{B}$ and $\omega'=\omega|_X (=d\alpha|_X)$. Let $R'$ be the Reeb vector field of $\alpha'$ and let $\chi'$ be the Lioville vector field of $\omega'$ pointing out from $B$. To finish the proof, we need to check that at a given point $p \in \partial X=B$ the orientation on $T_p \, B$ given by the form $\alpha'_p \wedge (d\alpha'_p)^{\wedge n-1}$ coincides with the one induced by the orientation on $T_p\,X$ given by the volume form $(\omega'_p)^{\wedge n}$:

Consider the contact structure $\xi':=\textrm{Ker} (\alpha')$ which is a symplectic subbundle (with rank $2n-2$) of $\xi=\textrm{Ker} (\alpha)$, and the decomposition (see \cite{Geiges}, for instance) $$T_p\,B=\langle R'_p \rangle \oplus \xi'_p.$$

From their definitions, we have $\omega'(\chi',R')>0$ which means that $\{\chi',R'\}$ is a nondegenerate pairing with respect to $\omega'$. Also since they are both transverse to $\xi'$, we get the decomposition $$T_p \,X=\langle \chi'_p \rangle \oplus \langle R'_p \rangle \oplus \xi'_p.$$

Choose a symplectic basis $\{u_1,v_1,...,u_{n-1},v_{n-1}\}$ for the symplectic subspace $(\xi'_p,\omega'_p)$ giving the orientation on $\xi'_p$ determined by $(\omega'_p)^{\wedge n-1}$, that is, we have $$(\omega'_p)^{\wedge n-1}(u_1,v_1,...,u_{n-1},v_{n-1})>0.$$ Since $(X,\omega')$ is a symplectic manifold and $\omega'_p(\chi'_p,R'_p)>0$, we get a symplectic basis $$\{\chi'_p\,, R'_p\,, u_1,v_1,...,u_{n-1},v_{n-1}\}$$ for the symplectic space $(T_p\, X,\omega'_p)$ giving the orientation on $T_p\, X$ determined by $(\omega')^{\wedge n}$, equivalently, $(\omega'_p)^{\wedge n}(\chi'_p\,, R'_p\,,u_1,v_1,...,u_{n-1},v_{n-1})>0$. Then the induced orientation on the subspace $T_p\,B \subset T_p\,X$ is determined by the oriented basis $$\{R'_p\,,u_1,v_1,...,u_{n-1},v_{n-1}\}$$ ($\chi'_p$ is outward pointing normal direction at $p\in B=\partial X$). Now, using the fact $\omega'=d\alpha'$, it is not hard to see that $$\alpha'_p \wedge (d\alpha'_p)^{\wedge n-1}(R'_p\,,u_1,v_1,...,u_{n-1},v_{n-1})>0.$$
\end{proof}


\section{Isotropic Setups and Weinstein Handles} \label{sec:Isotropic Setups and Weinstein Handles}

In this section we briefly recall the isotropic setups and Weinstein handles introduced in \cite{Weinstein}. Using them we will continue to study compatible exact Lefschetz fibrations in the next section.

\subsection{Isotropic setups}
Let $(M,\xi=\textrm{Ker} (\alpha))$ be a $(2n+1)$-dimensional contact
manifold. Any subbundle $\eta$ of the symplectic bundle $(\xi,d\alpha)$ has a symplectic orthogonal
$\eta^{\perp'} \subset \xi$. Therefore, if $Y$ is an isotropic submanifold of $M$, then $d\alpha |_Y=0$ (as $\alpha |_Y=0$), and so $$TY \subset (TY)^{\perp'} \subset \xi$$
from which we obtain the quotient bundle,
$$CSN(M,Y) = (TY)^{\perp'} /\,T Y$$
which is called the \emph{conformal symplectic normal bundle} of $Y$. Moreover, if $N(M,Y)$ denotes the normal bundle of $Y$ in $M$, then we have the decomposition

\begin{center}
$N(M,Y) = TM|_Y /\,TY\cong TM|_Y /\,\xi_Y \oplus \,\xi_Y /(TY)^{\perp'} \oplus (TY)^{\perp'}/\,TY$ \\
 $\cong \langle R_Y \rangle \oplus  T^*Y \oplus CSN(M,Y)$
\end{center}

\noindent where $R_Y$ is the Reeb vector-field $R$ of $\alpha$ restricted to $Y$. If we further assume that $Y$ is a sphere, then $\langle R_Y \rangle \oplus  T^*Y$ has a naturally trivialization. Hence, as pointed out in \cite{Weinstein}, any given trivialization of $CSN(M,Y)$ determines a \emph{framing} on $Y$ (that is, the trivialization of the normal bundle $N(M,Y)$), and the latter can be used to perform a surgery on $M$ along $Y$. Moreover, the resulting contact structure on the surgered manifold agrees with that of $M$ away from $Y$. Such an elementary surgery can be achieved also by attaching a Weinstein handle by making use of ``isotropic setups'' which we recall next.

A quintuple of the form $(P,\omega,\chi,M,Y)$ is called an \emph{isotropic setup} if $(P,\omega)$ is a symplectic manifold, $\chi$ is a Liouville vector field, $M$ is a hypersurface transverse to $\chi$ (hence a contact manifold), and $Y$ is an isotropic submanifold of $M$. The following proposition is the basic tool enabling us to attach Weinstein handles.

\begin{proposition}[\cite{Weinstein}] \label{prop:isotropic_setup}
Let $(P_1,\omega_1,\chi_1,M_1,Y_1), (P_2,\omega_2,\chi_2,M_2,Y_2)$ be two isotropic setups. Suppose that a given diffomorphism $Y_1\to Y_2$ is covered by a symplectic bundle isomorphism $$CSN(M_1,Y_1) \to CSN(M_2,Y_2).$$ Then there exist neighborhoods $U_j$ of $Y_j$ in $M_j$ and an isomorphism of isotropic setups $$\phi:(U_1,\omega_1|_{U_1},\chi_1|_{U_1},M_1 \cap U_1,Y_1) \to (U_2,\omega_2|_{U_2},\chi_2|_{U_2},M_2 \cap U_2,Y_2)$$ which restricts to the given map $Y_1 \to Y_2$, and induces the given bundle isomorphism.
\end{proposition}

\subsection{Weinstein handles}
Denote the coordinates on $\R^{2n+2}=\R^{2(n+1)}$ by $$(x_0,y_0,x_1,y_1,...,x_{n},y_{n})$$ and consider the standard symplectic structure on $\R^{2n+2}$ as $$\omega_0=\sum_{j=0}^n dx_j \wedge dy_j.$$ We will focus on two special Weinstein handles that we need for the present paper. Namely, let $H_n$ and $H_{n+1}$ be the $(2n+2)$-dimensional Weinstein handles in $\R^{2n+2}$ with indexes $n$ and $n+1$, respectively. These handles are defined as follow: Consider
$$\chi_n=-\dfrac{x_0}{2}\dfrac{\partial}{\partial x_0}-\dfrac{y_0}{2}\dfrac{\partial}{\partial y_0}+\sum_{j=1}^n \left(-2x_j\dfrac{\partial}{\partial x_j}+y_j\dfrac{\partial}{\partial y_j}\right), \quad \chi_{n+1}=\sum_{j=0}^n \left(-2x_j\dfrac{\partial}{\partial x_j}+y_j\dfrac{\partial}{\partial y_j}\right)$$
which are the negative gradient vector fields of the Morse functions
$$f_n=\dfrac{x_0^2}{4}+\dfrac{y_0^2}{4}+\sum_{j=1}^n \left(x_j^2-\dfrac{1}{2}y_j^2 \right), \quad f_{n+1}=\sum_{j=0}^n \left(x_j^2-\dfrac{1}{2}y_j^2 \right)$$
respectively. We have the contractions $\alpha_k=\iota_{\chi_k} \omega_0$, for $k=n,n+1$, given as
$$\alpha_n=-\dfrac{x_0}{2}\,dy_0+\dfrac{y_0}{2}\,dx_0+\sum_{j=1}^n \left(-2x_jdy_j-y_jdx_j\right), \quad \alpha_{n+1}=\sum_{j=0}^n \left(-2x_jdy_j-y_jdx_j\right)$$
from which we compute that $\mathcal{L}_{\chi_k}\omega_0=d(\iota_{\chi_k} \omega_0)=-\omega_0$. Therefore, $\chi_n, \chi_{n+1}$ are both Liouville vector fields of $\omega_0$. Next, consider the unstable manifold $$E_-^k=\{x_0=\cdots=x_n=y_0=\cdots=y_{n-k}=0\},$$ and the hypersurface $X_-=f_k^{-1}(-1)$ which is of contact type. The pull back of $\alpha_k$ on $E_-^k$ is zero, and so the descending sphere $$\mathcal{S}^{k-1}=E_-^k \cap X_-$$ is isotropic (Legendrian if $k=n+1$) in the contact manifold $(X_-,\alpha_k|_{X_-})$. Similarly, we have the stable manifold $E_+^{2n+2-k}=\{y_{n-k+1}=\cdots=y_{n}=0\}$ and the hypersurface $X_+=f_k^{-1}(1)$ intersecting each other along the ascending sphere $$S^{2n+1-k}=E_+^{2n+2-k} \cap X_+$$ which is a submanifold of the contact manifold $(X_+,\alpha_k|_{X_+})$.

\begin{figure}[ht] \label{fig:Weinstein_handle}
\begin{center}
\includegraphics{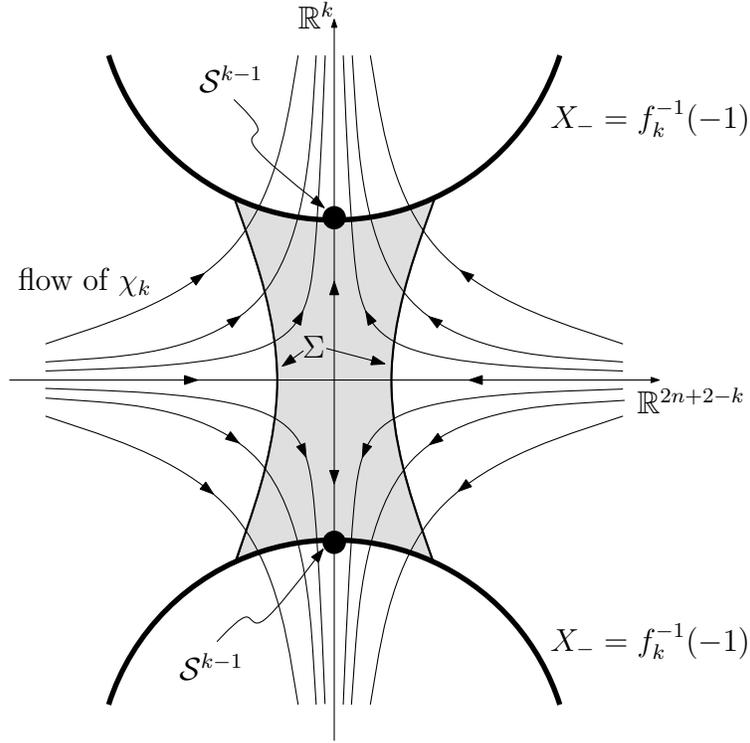}
\caption{Weinstein handle $H_k$ (shaded) and the flow of $\chi_k$ transverse to $\partial H_k$.}
\end{center}
\end{figure}

\vspace{1cm}

The Weinstein handle $H_k$ is the region bounded by a neighborhood (which can be taken arbitrarily small) of the descending sphere $\mathcal{S}^{k-1}$ in $X_-$ together with a connecting manifold $\Sigma \approx S^{2n+1-k} \times D^k$ depicted in Figure 1. It follows (see \cite{Weinstein}) that we can choose $\Sigma$ in such a way that $\chi_k$ is everywhere transverse to the boundary $\partial H_k$. Now we state the main theorem of \cite{Weinstein} which tells us, in particular, when we can attach Weinstein handles and how the symplectic structure extends over the handle.

\begin{theorem}[\cite{Weinstein}] \label{thm:Weinstein_main_theorem}
Let $Y$ be an isotropic sphere in the contact manifold $M$ with a trivialization of $CSN(M,Y)$. Let $M'$ be the manifold obtained from $M$ by elementary surgery along $Y$. Then the elementary cobordism $P$ from $M$ to $M'$ obtained by attaching a Weinstein handle to $M \times [0,1]$ along a neighborhood of $Y$ carries a symplectic structure and a Liouville vector field which is transverse to $M$ and $M'$. The contact structure induced on $M$ is the given one, while that on $M'$ differs from that on $M$ only on the spheres where the surgery takes place.
\end{theorem}

One important fact about gluing symplectic manifolds is not mentioned rigorously before (at least in \cite{Weinstein}). For our purposes it is convenient to state it as a lemma:

\begin{lemma} \label{lem:glueing_exact_forms}
Gluing two exact symplectic manifolds using an isomorphism of isotropic setups results in an exact symplectic manifold.
\end{lemma}

\begin{proof}
Let $(P_1,\omega_1,\alpha_1)$ and $(P_2,\omega_2,\alpha_2)$ be two exact symplectic manifold, and suppose that (as in Proposition \ref{prop:isotropic_setup}) there exists an isomorphism of isotropic setups
$$\phi:(U_1,\omega_1|_{U_1},\chi_1|_{U_1},M_1 \cap U_1,Y_1) \to (U_2,\omega_2|_{U_2},\chi_2|_{U_2},M_2 \cap U_2,Y_2)$$ which restricts to a given map $Y_1 \to Y_2$, and induces a given bundle isomorphism $$CSN(M_1,Y_1) \to CSN(M_2,Y_2).$$
Let $P$ be the manifold obtained by gluing $(P_1,\omega_1,\alpha_1)$ and $(P_2,\omega_2,\alpha_2)$ using the isomorphism $\phi$. This exactly means that along the gluing region we are gluing $\omega_i$'s, $\chi_i$'s (and so $\alpha_i$'s) together using $\phi$. Therefore, on the gluing region either of $(\omega_1,\alpha_1,\chi_1)$ or $(\omega_2,\alpha_2,\chi_2)$ defines an exact symplectic structure. Observe that on $P \setminus P_2$ (resp. on $P \setminus P_1$) the triple $(\omega_1,\alpha_1,\chi_1)$, (resp. $(\omega_2,\alpha_2,\chi_2)$) defines an exact symplectic structure. Hence, $P$ equips with the exact symplectic structure which we write as $(\omega_1 \cup_{\phi} \omega_2, \alpha_1 \cup_{\phi} \alpha_2, \chi_1 \cup_{\phi} \chi_2)$.
\end{proof}


\section{Convex Stabilizations} \label{sec:Convex Stabilization}

Our observation via isotropic setups and Weinstein handles is the fact that we can perform certain positive stabilizations, which will be called ``convex stabilizations'', in the category of compatible exact Lefschetz fibrations. Convex stabilizations will be defined explicitly at the end of the section where a summary of results and some corollaries are also presented in this new terminology. The main theorem of this section is

\begin{theorem} \label{thm:Legendrian_stabilization_gives_exact_Lefs_fib}
Any positive stabilization of a compatible exact Lefschetz fibration along a properly embedded Legendrian disk is also a compatible exact Lefschetz fibration.
\end{theorem}

\begin{proof}
Let $(\pi,E,\omega,\alpha,\chi,X,h)$ be a compatible exact Lefschetz fibration. We have already checked in the proof of Theorem \ref{thm:openbooks_Lefs.fibs} that a positive stabilization $\mathcal{S_{LF}}[(\pi,E,X,h);L]$ of the underlying Lefschetz fibration $(\pi,E,X,h)$ is an another Lefschetz fibration which we denoted by $(\pi',E',X',h')$. So all we need to check is that the exact symplectic structure $(\omega,\alpha,\chi)$ extends over the handles $H',H''$ which we used to construct $(\pi',E',X',h')$ so that we get an exact symplectic structure $(\omega',\alpha',\chi')$ on $E'$.

At this point one should ask why the Legendrian disk $L$ given on a page $X$ of the boundary exact open book (which carries $\xi=\textrm{Ker} (\alpha|\,_{\partial E})$ by Theorem \ref{thm:Exact_Openbooks_Support}) is also Lagrangian on the page $(X,d\alpha)$ (so that $\mathcal{S_{LF}}[(\pi,E,X,h);L]$ makes sense). We can check this as follows: From the basic equality
$$d\alpha(u,v)=\mathcal{L}_u\alpha(v)- \mathcal{L}_v\alpha(u)+ \alpha([u,v])$$
we immediately see that $d\alpha(u,v)=0$ for all $u,v \in TL$ (see Chapter III in \cite{Blair} for a discussion on integrable submanifolds of contact structures). This shows that $L$ is Lagrangian on the page $(X,d\alpha)$.

Consider the $2n$-dimensional Weinstein handle $H$ (of index $n$) used in Definition \ref{def:Stabilization_Lefschetz_Fibration} and in the proof of Theorem \ref{thm:openbooks_Lefs.fibs}. Taking the coordinates on $\R^{2n} \supset H$ as $(x_1,y_1,...,x_{n},y_{n})$, we can symplectically embed $H$ into $H_n$ by the map $$(x_1,y_1,...,x_{n},y_{n}) \to (0,0,x_1,y_1,...,x_{n},y_{n}).$$ Indeed, we can trivially fiber $H_n$ over $D^2$ with fibers diffeomorphic to $H$ by constructing it in a different way as follows: Our new model for $H_n$ will be $H\times D^2$. Consider the standard symplectic form $\omega_H=\sum_{j=1}^n dx_j \wedge dy_j$ on $H$ whose Liouville vector field $\chi_H$, the corresponding Morse function $f_H$ and the contraction $\alpha_H=\iota_{\chi_H}\omega_H$ are $$\chi_H=\sum_{j=1}^n \left(-2x_j\dfrac{\partial}{\partial x_j}+y_j\dfrac{\partial}{\partial y_j}\right), f_H=\sum_{j=1}^n \left(x_j^2-\dfrac{y_j^2}{2} \right), \alpha_H=\sum_{j=1}^n \left(-2x_jdy_j-y_jdx_j\right).$$

Let $(r,\theta)$ be the radial and the angle coordinates on $D^2$-factor in $H\times D^2$. If $pr_1$ (resp. $pr_2$) denotes the projection onto $H$-factor (resp. $D^2$-factor), then, similar to the proof of Lemma \ref{lem:Exact_Lefschetz_Fibration}, the modification $$\omega^0:=pr_1^*(\omega_H)+pr_2^*(rdr \wedge d\theta)$$ is a symplectic form on the total space $H_n=H\times D^2$ of the fibration $pr_2:H_n \to D^2$, and indeed is equivalent to the standard symplectic form $\omega_0$.
Considering $\chi_H$ and $-r/2 \,\partial / \partial r$ as vector fields in $T(H \times D^2)=TH \times TD^2$, it is straightforward to check that $$\chi^0:=\chi_H - r/2 \,\partial / \partial r$$ is the Liouville vector field of $\omega^0$ (satisfying $\mathcal{L}_{\chi^0}\omega^0=-\omega^0$) which gives the contraction $$\alpha^0:=\iota_{\chi^0}\omega^0=\alpha_H -r^2/2 \, d\theta.$$
Note that $\chi_H$ is transverse to $\partial_h H_n=\partial H \times D^2$ and $-r/2 \,\partial / \partial r$ is transverse to $\partial_v H_n=H \times S^1$, and so $\chi^0$ is everywhere transverse to $\partial H_n$. It follows that each fiber $$H_z:=pr_2^{-1}(z)\approx H \quad (z\in D^2)$$ is an $2n$-dimensional Weinstein handle (of index $n$) and equips with the exact symplectic form $\omega^0_z:=\omega^0|_{H_z}$ with the primitive $\alpha^0_z:=\alpha^0|_{H_z}$ and whose Liouville vector field $\chi^0_z:=\chi^0|_{H_z}$ is transverse to $\partial H_z$ and satisfies $\iota_{\chi^0_z} \omega^0_z=\alpha^0_z $.

As a result, we obtain a trivial (no singular fibers) Lefschetz fibration $(pr_2,H_n,H,\textrm{id})$ over $D^2$. One should note that this is not a compatible exact Lefschetz fibration because neither $\partial H_n$ nor $\partial H_z$ is convex, but it can be glued to a compatible exact Lefschetz fibration along the convex part, which we will denote by $\partial^{CX} H_n$, of its boundary to construct a new compatible exact Lefschetz fibration as we will see below. To describe $\partial^{CX} H_n$, we first observe that the boundary of $H$ is decomposed into its convex and concave parts as $$\partial H=\partial^{CX}H \cup \partial^{CV}H$$ where $\partial^{CX}H \approx \mathcal{S}^{n-1}\times D^n$ is the tubular neighborhood of descending sphere $\mathcal{S}^{n-1}$ in the hypersurface $f_H^{-1}(-1)$ from which $\chi_H$ points outward, and $\partial^{CV}H=\Sigma_H \approx S^{n-1}\times D^n$ is the connecting manifold from which $\chi_H$ points inward. Then we get the decomposition
$$\partial H_n=(\partial H \times D^2) \cup (H \times S^1)= (\partial^{CX}H \times D^2)\cup (\partial^{CV}H\times D^2) \cup (H \times S^1)$$
from which we deduce that $$\partial^{CX}H_n=\partial^{CX}H \times D^2  \quad \textrm{and} \quad \partial^{CV}H_n=(\partial^{CV}H \times D^2) \cup (H \times S^1).$$ An easy way to understand this decomposition is given schematically in Figure 2.

\begin{figure}[ht] \label{fig:Schematic_Weinstein_handle}
\begin{center}
\includegraphics{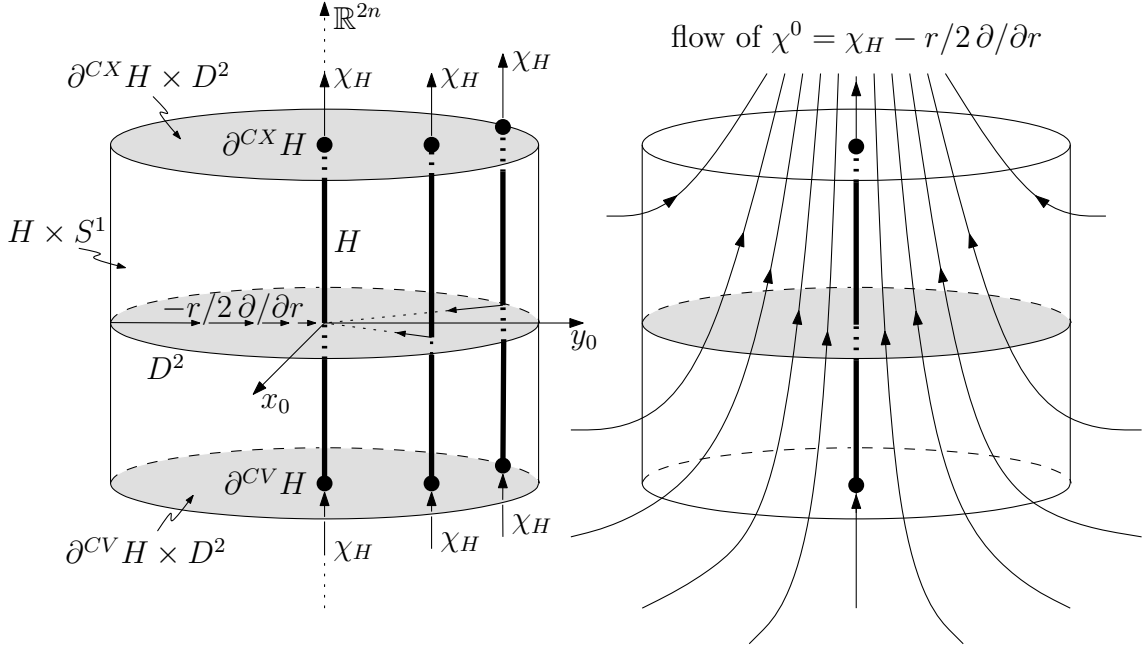}
\caption{A schematic picture of the convex and concave parts of $\partial H_n$ and the flow of $\chi^0=\chi_H -r/2 \,\partial / \partial r$ in $\R^{2n} \times \R^2$}
\end{center}
\end{figure}

\begin{lemma} \label{lem:H'_replaced_by_H_n}
The handle $H'$ can be replaced by the handle $H_{n}$. Moreover, the exact symplectic structure $(\omega,\alpha,\chi)$ on $E$ extends over the handle $H_n$.
\end{lemma}

\begin{proof}
We will replace the handle $H'= H \times D^2$ used in the proof of Theorem \ref{thm:openbooks_Lefs.fibs} with the Weinstein handle $H_n= H \times D^2$. Here we do the replacement in such a way that the new fiber $\tilde{E}_z \approx X'$ over $z\in D^2$ is obtained from the fiber $X=E_z$ by attaching the Weinstein handle $H_z$ along the Legendrian sphere $S_z:=S^{n-1} \subset \partial E_z$ which we consider as a copy of the boundary $\partial L$ of the Legendrian (and so Lagrangian) ball $L$ of the stabilization. More precisely, we proceed as follows:

As $S_z$ is Legendrian in $(\partial E_z,\alpha_z)$, its conformal symplectic normal bundle $CSN(\partial E_z,S_z)$ is zero (i.e., has rank zero). Similarly, the descending sphere $\mathcal{S}_z^{n-1}$ is Legendrian in $(\partial^{CX} H_z,\alpha^0_z)$ and so $CSN(\partial^{CX} H_z,\mathcal{S}_z^{n-1})$ is also zero. Therefore, by Proposition \ref{prop:isotropic_setup}, we can find neighborhoods $U_z$ of $\mathcal{S}_z^{n-1}$ in $H_z$ and $V_z$ of $S_z$ in $E_z$ and an isomorphism of isotropic setups
$$\phi_z:(U_z,\omega^0_z|_{U_z},\chi^0_z|_{U_z},\partial^{CX} H_z \cap U_z,\mathcal{S}_z^{n-1}) \to (V_z,\omega_z|_{V_z},\chi_z|_{V_z},\partial E_z \cap V_z,S_z) $$
which restricts to the map $f_z:\mathcal{S}_z^{n-1} \to S_z$ given in stage (I) of Definition \ref{def:Stabilization_Lefschetz_Fibration}. (Here $f_z$ is the embedding of the attaching sphere of $H_z$.) Now using Theorem \ref{thm:Weinstein_main_theorem} we attach each $H_z$ to corresponding $E_z$ using the isomorphism $\phi_z$ and obtain the new fiber $X'$ equipped with exact symplectic structure $$(\tilde{\omega}_z,\tilde{\alpha}_z,\tilde{\chi}_z):=(\omega_z \cup_{\phi_z} \omega^0_z, \alpha_z \cup_{\phi_z} \alpha^0_z, \chi_z \cup_{\phi_z} \chi^0_z).$$ (Note that $\omega_z \cup_{\phi_z} \omega^0_z=d(\alpha_z \cup_{\phi_z} \alpha^0_z)$ and we use Lemma \ref{lem:glueing_exact_forms} to obtain this structure.)

Next, we fix a copy of $S_{z_0} \subset \partial E_{z_0}$ (with $z_0 \in \textrm{int} D^2$) of the Legendrian sphere $\partial L$ in a fixed regular fiber $E_{z_0}$. Since $z_0$ is a regular value of $\pi$, we may assume that $\partial E_{z_0}$ is the binding of the (exact) open book induced by $\pi$. Since this open book carries the contact structure $\xi=\textrm{Ker}(\alpha)$ on $\partial E$ (which we know by Theorem \ref{thm:Exact_Openbooks_Support}), the binding $(\partial E_{z_0},\xi_{z_0}:=\textrm{Ker}(\alpha_{z_0}))$ is a contact submanifold manifold of $(\partial E,\xi)$, and so $\xi_{z_0}$ is a subbundle of $(\xi|_{\partial E_{z_0}},d\alpha|_{\partial E_{z_0}})$ and we have (see for instance \cite{Geiges}) $$T\partial E|_{\partial E_{z_0}}=T\partial E_{z_0} \oplus (\xi_{z_0})^{\perp'}$$ where $(\xi_{z_0})^{\perp'}=CSN(\partial E,\partial E_{z_0})$ is the symplectically orthogonal complement of $\xi_{z_0}$ in $(\xi|_{\partial E_{z_0}},d\alpha|_{\partial E_{z_0}})$. ($(\xi_{z_0})^{\perp'}$ is also called the conformal symplectic normal bundle of $\partial E_{z_0}$ in $\partial E$.) The latter equality implies that $CSN(\partial E,\partial E_{z_0})$ can be identified with the classical normal bundle $N(\partial E,\partial E_{z_0})$ of $\partial E_{z_0}$ in $\partial E$. But we know, by definition of open books, that the binding has a trivial normal bundle, so $CSN(\partial E,\partial E_{z_0})=\partial E_{z_0} \times D^2$. Then from the inclusions $S_{z_0} \subset \partial E_{z_0}\subset \partial E$ we have $$CSN(\partial E,S_{z_0})=CSN(\partial E,\partial E_{z_0})|_{S_{z_0}} \oplus CSN(\partial E_{z_0},S_{z_0})=S_{z_0} \times D^2.$$
(Recall $CSN(\partial E_{z_0},S_{z_0})$ is zero as $S_{z_0}$ is Legendrian in $\partial E_{z_0}$.)\\

For $\mathcal{S}^{n-1}$, we follow not the same but similar lines: We fix a copy $\mathcal{S}^{n-1}_{z_0} \subset \partial^{CX} H_{z_0}$ in a fixed fiber $H_{z_0}$ (with $z_0 \in \textrm{int} D^2$). The restriction of $\alpha^0_{z_0}$ onto $\partial^{CX} H_{z_0}$ is a contact form making $\partial^{CX} H_{z_0}$ convex, and $\mathcal{S}^{n-1}_{z_0}$ is Legendrian in $(\partial^{CX} H_{z_0}, \alpha^0_{z_0}|_{\partial^{CX} H_{z_0}})$. So we have $$CSN(\partial^{CX} H_{z_0},\mathcal{S}^{n-1}_{z_0})=0.$$
Also $(\partial^{CX} H_{z_0},\alpha^0_{z_0}|_{\partial^{CX} H_{z_0}})$ is a contact submanifold manifold of $(\partial^{CX} H_n,\alpha^0|_{\partial^{CX} H_n})$. Then from the inclusions $\mathcal{S}^{n-1}_{z_0} \subset \partial^{CX} H_{z_0} \subset \partial^{CX} H_n$ we see that
$$CSN(\partial^{CX} H_n,\mathcal{S}^{n-1}_{z_0})=CSN(\partial^{CX} H_n,\partial^{CX} H_{z_0})|_{\mathcal{S}^{n-1}_{z_0}} \oplus CSN(\partial^{CX} H_{z_0},\mathcal{S}^{n-1}_{z_0})=\mathcal{S}^{n-1}_{z_0} \times D^2.$$

Now we will show that all the above individual attachments are indeed pieces of the attachment of the Weinstein handle $H_n$ to $E$ along $S_{z_0}$ by finding an isotropic setup which agrees with each individual fiber-wise gluing. To this end, note that we have the map $f_{z_0}:\mathcal{S}_{z_0}^{n-1} \to S_{z_0}$ given in stage (I) of Definition \ref{def:Stabilization_Lefschetz_Fibration}.
Define the map $$\Psi:CSN(\partial^{CX} H_n,\mathcal{S}^{n-1}_{z_0})=\mathcal{S}^{n-1}_{z_0} \times D^2 \longrightarrow S_{z_0} \times D^2=CSN(\partial E,S_{z_0})$$ by the rule $$\Psi(p,z)=(f_{z_0}(p),z).$$
Clearly, $\Psi$ is a bundle map and covers $f_{z_0}$, and so by Proposition \ref{prop:isotropic_setup}, we can find neighborhoods $U$ of $\mathcal{S}_{z_0}^{n-1}$ in $H_n$ and $V$ of $S_{z_0}$ in $E$ and an isomorphism of isotropic setups
$$\phi_n:(U,\omega^0|_{U},\chi^0|_{U},\partial^{CX} H_n \cap U,\mathcal{S}_{z_0}^{n-1}) \to (V,\omega|_{V},\chi|_{V},\partial E \cap V,S_{z_0}) $$
which restricts to $f_{z_0}$ and the bundle map $\Psi$. We may assume that $\partial^{CX} H_n \cap U=\partial^{CX} H_n$, that is, $\phi_n$ attaches $H_n$ to $E$ along the whole convex part $\partial^{CX} H_n= \partial^{CX} H \times D^2$ of its boundary. Now consider the boundaries
$$\partial^{CX} H_n=\partial^{CX} H_{z_0} \times D^2 \quad \textrm{and} \quad \partial_h E=\bigcup_{z\in D^2} \partial E_z = \partial E_{z_0} \times D^2.$$ For each $z \in D^2$, by attaching $H_z$ to $E_z$ using $\phi_z$, we glue $\partial^{CX} H_{z_0} \times \{z\} \in \partial^{CX} H_n$ with $\partial E_{z_0} \times \{z\} \in \partial_h E$ and also we map $\mathcal{S}_{z}^{n-1}$ onto $S_{z}$ by $f_z$. Therefore, attaching all $H_z$'s to $E_z$'s along $\phi_z$'s defines a smooth map $\mathcal{S}^{n-1}_{z_0} \times D^2 \longrightarrow S_{z_0} \times D^2$ which is identity on the $D^2$-factor and maps $\mathcal{S}^{n-1}_{z_0}$ onto $S_{z_0}$ via $f_{z_0}$, and so it coincides with $\Psi$. Hence, we conclude that overall effect of attaching all $H_z$'s to $E_z$'s using $\phi_z$'s on $E$ is equivalent to attaching Weinstein handle $H_n$ to $E$ using $\phi_n$.

By Lemma \ref{lem:glueing_exact_forms} we know that the resulting manifold $\tilde{E}:=E \cup_{\phi_n} H_n$ has an exact symplectic structure $(\tilde{\omega},\tilde{\alpha},\tilde{\chi})$ obtained by gluing those on $E$ and $H_n$. In other words,
$$(\tilde{\omega},\tilde{\alpha},\tilde{\chi})=(\omega \cup_{\phi_n} \omega^0, \alpha \cup_{\phi_n} \alpha^0, \chi \cup_{\phi_n} \chi^0).$$
Also, clearly, $\pi$ extends over $H_n$ and we get a Lefschetz fibration $\tilde{\pi}:\tilde{E} \to D^2$ with regular fiber $X'$ and monodromy $h$ (original $h$ which is trivially extended over $H$).
To check that $(\tilde{\omega},\tilde{\alpha},\tilde{\chi})$ restricts to $(\tilde{\omega}_z,\tilde{\alpha}_z,\tilde{\chi}_z)$ on each new regular fiber $\tilde{E}_z \approx X'$, we proceed as follows: For each $z \in D^2$, by taking $U_z$ (resp. $V_z$) small enough, we can guarantee that the union $$\bigcup_{z\in D^2} U_z \quad \left(\textrm{resp.} \bigcup_{z\in D^2} V_z\right)$$ lies in the collar neighborhood of $\partial^{CX} H_n$ (resp. $\partial_h E$) where we have the local triviality condition (as described in the definition of exact symplectic fibration). By using these local trivialities, we combine all the exact symplectic structures $(\omega_z \cup_{\phi_z} \omega^0_z, \alpha_z \cup_{\phi_z} \alpha^0_z, \chi_z \cup_{\phi_z} \chi^0_z)$ together, and surely the resulting structure must be $(\omega \cup_{\phi_n} \omega^0, \alpha \cup_{\phi_n} \alpha^0, \chi \cup_{\phi_n} \chi^0)$ because $(\omega_z,\alpha_z,\chi_z)$'s (resp. $(\omega^0_z,\alpha^0_z,\chi^0_z)$'s) patch together and give $(\omega,\alpha,\chi)$ (resp. $(\omega^0,\alpha^0,\chi^0)$). This completes the proof of Lemma \ref{lem:H'_replaced_by_H_n}.
\end{proof}

So far we have constructed a compatible exact Lefschetz fibration $(\tilde{\pi},\tilde{E},\tilde{\omega},\tilde{\alpha},\tilde{\chi},X',h)$ on $$\tilde{E}:=E \cup_{\phi_n} H_n,$$ in other words, we extended  $(\omega,\alpha,\chi)$ over the handle $H'$ by showing that $H'$ can be replaced by $H_n$. Next, we want to extend $(\tilde{\omega},\tilde{\alpha},\tilde{\chi})$ over $H''$ by showing that $H''$ can be replaced by the Weinstein handle $H_{n+1}$.

\begin{remark} \label{rmk:Reversing_Louville_vec_field}
Although Weinstein handles are attached along the convex part of their boundaries (according to the convention of present paper which coincides with the one in \cite{Weinstein}), we actually need to reverse the direction of the Liouville vector field of the Weinstein handle when it is being attached to a convex boundary of a symplectic manifold. Otherwise it is impossible to match the Liouville directions of the symplectizations used in the gluing. Since we have attached $H_n$ to $E$ along the whole $\partial^{CX}H_n$ by matching $-\chi^0=-\chi_H+r/2\,\partial / \partial r$ with $\chi$, we have to now consider $$\partial^{CV} H_n=(\partial^{CV}H \times D^2) \cup (H \times S^1)$$ as a subset of the convex part of the boundary $\partial \tilde{E}$.
\end{remark}

\begin{lemma} \label{lem:H''_replaced_by_H_n+1}
The Lefschetz handle $H''$ can be replaced by the Weinstein handle $H_{n+1}$.
\end{lemma}

\begin{proof}
Recall that $H''$ is attach to $\tilde{E}$ along the Lagrangian $n$-sphere $S$ on a page $(X',d\tilde{\alpha}|_{X'})$ of the boundary exact open book $(X',h)$ carrying the contact structure $\tilde{\xi}=\textrm{Ker} (\tilde{\alpha}|_{\,\partial \tilde{E}})$ on $\partial \tilde{E}$. Say $S \subset \tilde{E}_{\theta_0}\,(\approx X')$ for some $\theta_0 \in S^1=\partial D^2$. From its construction (given in Definition \ref{def:Stabilization_Lefschetz_Fibration}) and the notation introduced above, $S$ is the union
$$S=L \cup_{f_{\theta_0}} D$$ of the Lagrangian $n$-disk $L \in (E_{\theta_0},\omega_{\theta_0}=d\alpha_{\theta_0})$ and the Lagrangian core disk $D\,(\approx D^n)$ of the $2n$-dimensional Weinstein handle $(H_{\theta_0},d\alpha^0_{\theta_0})$. Note that $H_{\theta_0}\approx H$ is the fiber (over ${\theta_0} \in S^1$) of the trivial fibration  $H \times S^1$. By assumption $L$ is Legendrian in $(\partial \tilde{E}, \tilde{\xi})$. On the other hand, the contact form $\tilde{\alpha}|_{\partial \tilde{E}}$ restricts to a contact form $$\alpha_\sharp:=(\tilde{\alpha}|_{\,\partial \tilde{E}})|_{H \times S^1}=(\iota_{-\chi^0}\omega^0)|_{H \times S^1}=-\alpha^0|_{H \times S^1}=\frac{d\theta}{2}-\alpha_H=\frac{d\theta}{2} +\sum_{j=1}^n (2x_jdy_j+y_jdx_j)$$
on a convex part $H \times S^1 \subset \partial \tilde{E}$. Observe that the core disk $D \subset H_{\theta_0}=H \times \{{\theta_0}\}$ is given by the set $$\{x_1=x_2=\cdots =x_n=0, \quad \theta={\theta_0} \textrm{(constant)}\},$$ and so clearly $\alpha_\sharp=0$ on $D$ which means that $D$ is Legendrian in $(H \times S^1,\alpha_\sharp) \subset (\partial \tilde{E},\alpha|_{\partial \tilde{E}})$. Therefore, the $n$-sphere $S$ is also Legendrian in $(\partial \tilde{E},\alpha|_{\partial \tilde{E}})$ which implies that $CSN(\partial \tilde{E},S)=0$. Moreover, we also have $CSN(\partial^{CX}H_{n+1},\mathcal{S}^{n})=0$ as $\mathcal{S}^{n}$ is Legendrian in $(\partial^{CX}H_{n+1},\alpha_{n+1}|_{\partial^{CX}H_{n+1}})$ by definition. Then by Proposition \ref{prop:isotropic_setup}, we can find neighborhoods $U'$ of $\mathcal{S}^n$ in $H_{n+1}$ and $V'$ of $S$ in $\tilde{E}$ and an isomorphism of isotropic setups
$$\phi_{n+1}:(U',\omega_0|_{U'},\chi_{n+1}|_{U'},\partial^{CX} H_{n+1} \cap U',\mathcal{S}^{n}) \to (V',\tilde{\omega}|_{V'},\tilde{\chi}|_{V'},\partial \tilde{E} \cap V',S)$$
which restricts to the embedding $\phi:\mathcal{S}^{n} \to S$ determined by Definition \ref{def:Stabilization_Lefschetz_Fibration}. Now by Theorem \ref{thm:Weinstein_main_theorem} attaching $H_{n+1}$ to $\tilde{E}$ using $\phi_{n+1}$ results in an exact symplectic manifold $$E':=\tilde{E} \cup_{\phi_{n+1}} H_{n+1}$$ equipped with the exact symplectic data $$(\omega',\alpha',\chi')=(\tilde{\omega} \cup_{\phi_{n+1}} \omega_0, \tilde{\alpha} \cup_{\phi_{n+1}} \alpha_{n+1}, \tilde{\chi} \cup_{\phi_{n+1}} \chi_{n+1}).$$
Again we may assume that $\partial^{CX} H_{n+1} \cap U'=\partial^{CX} H_{n+1}$, that is, $\phi_{n+1}$ attaches $H_{n+1}$ to $\tilde{E}$ along the whole convex part $\partial^{CX} H_{n+1}$ of its boundary. Note that the step we just explained replaces $H''$ with the Weinstein handle $H_{n+1}$. From the bundle isomorphisms
$$\nu_1 \oplus \varepsilon \cong TS \oplus \varepsilon \cong T^*S \oplus (T\partial \tilde{E}/\tilde{\xi})|_S $$ we see that the framings on the normal bundle $N(\partial \tilde{E},S)$ which are used to attach $H''$ and $H_{n+1}$ coincide, and so attaching  $H''$ and $H_{n+1}$ are topologically the same. Therefore, we know by Theorem \ref{thm:Determination_Lefschetz_fibration} that when we add $H_{n+1}$ to $(\tilde{\pi},\tilde{E},\tilde{\omega},\tilde{\alpha},\tilde{\chi},X',h)$, we can extend the underlying topological Lefschetz fibration $(\tilde{\pi},\tilde{E},X',h)$ over $H_{n+1}$ and get the Lefschetz fibration $$\mathcal{S_{LF}}[(\pi,E,X,h);L]=(\pi',E',X',\delta_{(\phi,\phi')}\circ h)$$ where $\delta_{(\phi,\phi')}$ is the right-handed Dehn twist described in Definition \ref{def:Stabilization_Lefschetz_Fibration}. The proof of Lemma \ref{lem:H''_replaced_by_H_n+1} is now complete.
\end{proof}

To be able to say that we have constructed a compatible exact Lefschetz fibration $$(\pi',E',\omega',\alpha',\chi',X',\delta_{(\phi,\phi')}\circ h)$$ on $E'$, it remains to check that the exact symplectic structure $(\omega',\alpha',\chi')$ restricts to an exact symplectic structure on every new regular fiber $E'_z$. Note that this time we are not changing the diffeomorphism type of the regular fiber, that is $E'_z\approx\tilde{E}_z\approx X'$. The Weinstein handle $H_{n+1}$ is attached to $\tilde{E}$ along the neighborhood $$\partial \tilde{E} \cap V' \approx N(\tilde{E}_{\theta_0},S) \times [0,1]$$ of the attaching sphere $S \subset \tilde{E}_{\theta_0}$ in $\partial \tilde{E}$ where we identify the interval $[0,1]$ with a closed arc in $S^1\setminus \{pt\}$ such that $0<\theta_0<1$. Consider the mapping torus $$X' \times [0,1] / (x,0)\sim(h(x),1)$$ of the open book $(X',h)$ on $\partial \tilde{E}$ and the inclusion $$N(\tilde{E}_{\theta_0},S) \times [0,1] \subset X' \times [0,1].$$ Observe that attaching $H_{n+1}$ to $\tilde{E}$ along the attaching region $N(\tilde{E}_{\theta_0},S) \times [0,1]$ results in a new mapping torus $$X' \times [0,1] / (x,0)\sim((\delta_{(\phi,\phi')}\circ h)(x),1)$$  for the open book $(X',\delta_{(\phi,\phi')}\circ h)$ on the new boundary $(\partial E', \xi':=\textrm{Ker} (\alpha'|_{\partial E'}))$ obtained from the corresponding elementary (contact) surgery on $(\partial \tilde{E}, \tilde{\xi})$ along the Legendrian sphere $S$. To get this new mapping torus, we are just gluing two copies of $X'$ equipped with the exact symplectic structure $(\tilde{\omega}_{X'},d\tilde{\omega}_{X'},\tilde{\chi}_{X'})$ using the symplectomorphism $$\delta_{(\phi,\phi')}\circ h \in \textrm{Symp}(X',\tilde{\omega}_{X'}).$$
Therefore, attaching $H_{n+1}$ does not change the exact symplectic structures of regular fibers. But of course, it does change the structure of the Lefschetz fibration: Relative to $\tilde{\pi}:\tilde{E}\to D^2$, the new Lefschetz fibration $\pi':E'=\tilde{E} \cup H_{n+1} \to D^2$ has one more critical point (and so one more singular fiber) located at the origin in the Weinstein handle $H_{n+1}$.

We conclude that $(\omega',\alpha',\chi')$ restricts to
$$(\omega'_z,\alpha'_z,\chi'_z)=(\tilde{\omega}_z, \tilde{\alpha}_z, \tilde{\chi}_z)$$
on each regular fiber $E'_z\approx\tilde{E}_z\approx X'$ of $\pi'$. Hence, we have a compatible exact Lefschetz fibration $(\pi',E',\omega',\alpha',\chi',X',\delta_{(\phi,\phi')}\circ h)$ as claimed.
This finishes the proof of Theorem \ref{thm:Legendrian_stabilization_gives_exact_Lefs_fib}.
\end{proof}

We have the following consequence of Theorem \ref{thm:Legendrian_stabilization_gives_exact_Lefs_fib}:

\begin{corollary} \label{cor:Legendrian_stabilization_gives_exact_open_book}
Any positive stabilization of an exact open book along a properly embedded Legendrian disk is also an exact open book.
\end{corollary}

\begin{proof} By definition if $(X,h)$ is an exact open book, then there exist a compatible exact Lefschetz fibration $(\pi,E,\omega,\alpha,\chi,X,h)$ which induces $(X,h)$ on the boundary. Let $L$ be any properly embedded Legendrian (and so Lagrangian) disk in $(X,\omega)$. Then by Theorem \ref{thm:openbooks_Lefs.fibs} we know that the stabilization $\mathcal{S_{OB}}[(X,h);L]$ is induced by  $\mathcal{S_{LF}}[(\pi,E,X,h);L]$. Moreover, we know, by Theorem \ref{thm:Legendrian_stabilization_gives_exact_Lefs_fib}, that there exists a compatible exact Lefschetz fibration $(\pi',E',\omega',\alpha',\chi',X',h')$ with underlying topological Lefschetz fibration $\mathcal{S_{LF}}[(\pi,E,X,h);L]$. In particular, $\mathcal{S_{OB}}[(X,h);L]$ is induced by $(\pi',E',\omega',\alpha',\chi',X',h')$, and hence, it is exact by definition.
\end{proof}

After all, the following definitions make sense and fit into the frame very well.

\begin{definition} \label{def:Convex_Stabilizations}
\begin{itemize} \item[(i)] A \emph{convex stabilization} $\mathcal{S^{C}_{LF}}[(\pi,E,\omega,\alpha,\chi,X,h);L]$ of a compatible exact Lefschetz fibration $(\pi,E,\omega,\alpha,\chi,X,h)$ is defined to be the positive stabilization $\mathcal{S_{LF}}[(\pi,E,X,h);L]$ where $L$ is a properly embedded Legendrian disk on $X$.\\

\item[(ii)] A \emph{convex stabilization} $\mathcal{S^{C}_{OB}}[(X,h);L]$ of an exact open book $(X,h)$ is defined to be the positive stabilization $\mathcal{S_{OB}}[(X,h);L]$ where $L$ is a properly embedded Legendrian disk on $X$.
\end{itemize}
\end{definition}

The theorem that we state next can be considered as the exact symplectic version of Theorem \ref{thm:openbooks_Lefs.fibs}. It summarizes some of the results that we have shown in the language of convex stabilizations. The proof is a straight forward combination of previous statements and definitions, and so will be omitted.

\begin{theorem} \label{thm:exactopenbooks_exactLefs.fibs}
$\mathcal{S^{C}_{LF}}[(\pi,E,\omega,\alpha,\chi,X,h);L]$ induces the (exact) open book $\mathcal{S^{C}_{OB}}[(X,h);L]$. Conversely, if an (exact) open book $(X,h)$ is induced by $(\pi,E,\omega,\alpha,\chi,X,h)$, then any convex stabilization $\mathcal{S^{C}_{OB}}[(X,h);L]$ of $(X,h)$ is induced by $\mathcal{S^{C}_{LF}}[(\pi,E,\omega,\alpha,\chi,X,h);L]$. \qed
\end{theorem}

Combining the results we get so far, we know that a convex stabilization of a compatible exact Lefschetz fibration produces an another compatible exact Lefschetz fibration on a manifold which has the same diffeomorphism type with the original one. One can see that these manifolds have symplectomorphic completions (recall that given a strong symplectic filling $E$ of a contact manifold $M$, the \emph{completion} of $E$ is a noncompact symplectic manifold obtained from $E$ by gluing the positive end of the symplectization of $M$):

\begin{theorem} \label{thm:Symplectic_Canceling_Pair}
Let $(E',\omega',\alpha')$ be the total space of $\mathcal{S^{C}_{LF}}[(\pi,E,\omega,\alpha,\chi,X,h);L]$, Then $(E',\omega',\alpha')$ and $(E,\omega,\alpha)$ have symplectomorphic completions. In other words, the pair $$\{H_n,H_{n+1}\}$$ used in the construction of $\mathcal{S^{C}_{LF}}[(\pi,E,\omega,\alpha,\chi,X,h);L]$
is a symplectically canceling pair for the completions.
\end{theorem}

\begin{proof} We have already observed in the proof of Theorem \ref{thm:openbooks_Lefs.fibs} that $\{H_n,H_{n+1}\}$ is a canceling pair in smooth category (as the belt sphere of $H_n$ intersects the attaching sphere of $H_{n+1}$ transversely once). Moreover, Lemma 3.6b in \cite{Eliashberg} (see also Lemma 3.9 in \cite{Van Koert}) implies that two Weinstein handles form a symplectically canceling pair for the completions if they form a canceling pair in smooth category and their Morse-index difference is one. As a result, we conclude that $\{H_n,H_{n+1}\}$ is a symplectically canceling pair for the completions (for a more precise discussion see Lemma 3.11 in \cite{Van Koert}).
\end{proof}

As an immediate corollary, we have

\begin{corollary} \label{cor:Contactomorphic_Boundaries}
Let $\xi$ (resp. $\xi'$) be the contact structure on $\partial E$ (resp. $\partial E'$) induced by the exact symplectic structure of $(\pi,E,\omega,\alpha,\chi,X,h)$ (resp. $(\pi',E',\omega',\alpha',\chi',X',h')=\mathcal{S^{C}_{LF}}[(\pi,E,\omega,\alpha,\chi,X,h);L]$). Then $(\partial E,\xi)$ is contactomorhic $(\partial E',\xi')$. \hspace{2.7cm} $\square$
\end{corollary}

Finally, as an application, we verify a well-known result for the class of exact open books and their convex stabilizations. Namely,

\begin{corollary} \label{cor:convex_stab_respects_contact_str}
Let $\xi$ be a contact structure carried by an exact open book $(X,h)$. Then any convex stabilization $\mathcal{S^{C}_{OB}}[(X,h);L]$ of $(X,h)$ carries $\xi$.
\end{corollary}

\begin{proof} By assumption, there is a compatible exact Lefschetz fibration $(\pi,E,\omega,\alpha,\chi,X,h)$ which induces $(X,h)$ on the boundary. Note that, by Theorem \ref{thm:Exact_Openbooks_Support}, $(\omega,\alpha,\chi)$ induces $\xi$ on $\partial E$.  Theorem \ref{thm:exactopenbooks_exactLefs.fibs} implies that $\mathcal{S^{C}_{OB}}[(X,h);L]$ is induced by $\mathcal{S^{C}_{LF}}[(\pi,E,\omega,\alpha,\chi,X,h);L]$. Moreover, again by Theorem \ref{thm:Exact_Openbooks_Support}, we know that $\mathcal{S^{C}_{OB}}[(X,h);L]$  carries the contact structure induced by the exact symplectic structure on $\mathcal{S^{C}_{LF}}[(\pi,E,\omega,\alpha,\chi,X,h);L]$. Now the proof follows from Corollary \ref{cor:Contactomorphic_Boundaries}.
\end{proof}

\addcontentsline{toc}{chapter}{\textsc{References}}

\end{document}